\documentclass[10pt]{article}

\usepackage[latin1]{inputenc}
\usepackage[english]{babel}
\usepackage[T1]{fontenc}
\usepackage{lmodern}
\usepackage{graphicx, xcolor}
\usepackage{amsfonts}
\usepackage{stmaryrd}
\usepackage[nice]{nicefrac}
\usepackage{amsmath,amsthm,amssymb,mathabx}
\usepackage{enumerate}
\usepackage{hyperref}
\usepackage[top=2cm, bottom=2cm, left=2cm , right=2cm]{geometry}
\usepackage{caption}
\usepackage{wrapfig}

\usepackage{caption}
\newtheorem{thm}{Theorem}
\newtheorem{definition}{Definition}
\newtheorem{lem}{Lemma}
\newtheorem{prop}{Proposition}
\newtheorem{cor}{Corollary}

\def\indi{\mathrm{\hspace{0.2em}l\hspace{-0.55em}1}}

\def\Re{{\rm Re}}

\newcommand{\bs}[1]{\ensuremath{\boldsymbol{#1}}}

\newcommand{\sss}{\scriptscriptstyle}

\begin{document}
\title{Multiple drawing multi-colour urns by stochastic approximation\thanks{This new arxiv version (v6) corrects a mistake that we discovered in the previous versions of this paper (v1-5). The mistake was in Theorem 1$(a)$ and in the last sentence of Theorem~4. In this new version, Theorem 1$(a)$ has been corrected, and Theorem 4 has been deleted.}}
\author{
Nabil Lasmar\thanks{D\'epartement des Math\'ematiques, Institut Pr\'eparatoire aux \'Etudes d'Ing\'enieur, Monastir, Tunisia (\texttt{nabillasmar@yahoo.fr}).}, 
C\'ecile Mailler\thanks{University of Bath, Department of Mathematical Sciences, Claverton Down, BA2 7AY Bath, UK (\texttt{c.mailler@bath.ac.uk}).}~ and 
Olfa Selmi\thanks{D\'epartement des Math\'ematiques, Facult\'e des Sciences de Monastir, Monastir, Tunisia (\texttt{selmiolfa3@yahoo.fr}).}}
\maketitle

\begin{abstract}
A classical P\'olya urn scheme is a Markov process whose evolution is encoded by a replacement matrix $(R_{i,j})_{1\leq i,j\leq d}$.
At every discrete time-step, we draw a ball uniformly at random, denote its colour $c$, and replace it in the urn together with $R_{c,j}$ balls of colour $j$ (for all $1\leq j\leq d$).

We study multi-drawing P\'olya urns, 
where the replacement rule depends on the random drawing of a set of $m$ balls from the urn (with or without replacement). 
Many particular examples of this situation have been studied in the literature, but the only general results are by 
Kuba \& Mahmoud (ArXiv:1503.09069 and 1509.09053).
These authors prove second order asymptotic results in the $2$-colour case, 
under the so-called {\it balance} and {\it affinity} assumptions,
the latter being somewhat artificial.

The main idea of this work is to apply stochastic approximation methods to this problem, 
which enables us to prove analogous results to Kuba \& Mahmoud, 
but without the artificial {\it affinity} hypothesis,
and, for the first time in the literature, in the $d$-colour case ($d\geq 3$).
We also give some partial results in the two-colour non-balanced case,
the novelty here being that the only results for this case currently in the literature are for particular examples.
\end{abstract}

\section{Introduction}
%\addcontentsline{toc}{section}{\protect\numberline{}Introduction}%

%%%%
\subsection{Classical P\'olya urns}

P{\'o}lya urn schemes are the simplest example of stochastic processes with reinforcement.
Understanding these objects is thus the first step in understanding many more intricate models 
such as reinforced random walks, preferential attachment networks, interacting urn models, etc.
Furthermore, reinforced stochastic processes appear in a wide range of applications; 
in biology (e.g. ants walks, reinforced branching processes), 
in finance and clinical trials, and in computer science 
(preferential attachment networks are used to model the complex networks such as 
the internet, the World Wide Web, or social networks). 
We refer the reader to Pemantle's survey~\cite{Pemantle} 
for an overview on both mathematics and applications of reinforced processes.

The classical P{\'o}lya urn, first introduced by Eggenberger \& P{\'o}lya~\cite{EP23}, 
is described as follows: 
an urn contains initially one white ball and one black ball.
At each discrete time step, one picks a ball uniformly at random 
and replaces it in the urn together with another ball of the same colour.
Many generalisations of this model have been studied in the literature, namely 
urns with more than two colours, or with different, possibly random, replacement rules.

The methods used to study P\'olya urns are quite varied: 
since the seminal work of Athreya \& Karlin~\cite{AK}, 
one successful approach is to embed the urn process in continuous time using exponential clocks 
and then apply martingale arguments.
This method can be very effective - see, for example Janson~\cite{Janson04,Janson06}; 
the main difficulty is then to pull the results back into the discrete-time framework.

Since the work of Flajolet, Dumas \& Puyhaubert~\cite{FDP}, 
analytic combinatorics have been used to study P{\'o}lya's urn schemes 
(see Morcrette~\cite{Basile} and Morcrette \& Mahmoud~\cite{MM}).
The main advantage of this method is that, when successful, 
it gives results for fixed finite time and not only asymptotically 
when time goes to infinity as the embedding-in-continuous-time method does.
The major drawback is that this method is often non-tractable 
as it relies on solving a non-linear differential system.

Finally, stochastic approximation (or stochastic algorithms) provides a powerful toolbox to prove strong laws of large numbers 
and central limit theorems for the composition of a P{\'o}lya urn.
This method relies on discrete time martingale methods; 
a good introduction to stochastic approximation is the book by Duflo~\cite{Duflo}. 
The literature on stochastic algorithms is very wide; 
in the urn context, stochastic approximation has already been used, 
for example by Laruelle \& Pag\`es~\cite{LP} 
for the study of P\'olya urns with random replacement rules.
In this article, we apply the stochastic approximation results of Zhang~\cite{Zhang16} and Renlund~\cite{Renlund} 
to a model of multi-drawing P\'olya urns.

\subsection{Multi-drawing P\'olya urns}

In this article, we focus on the multi-drawing generalisation of P{\'o}lya urns:
instead of choosing the replacement rule according to the random drawing of one ball in the urn, 
one picks at random a handful of balls (more precisely a fixed number $m$ of balls), 
and the replacement rule then depends on the colours of those $m$ balls.
This model has numerous applications, such as
degrees in increasing trees and preferential attachment networks (see Kuba \& Sulzbach~\cite{KS}), 
leaves in random circuits (see~\cite{KS}) 
and tournaments {\it \`a la} {\it rock-paper-scissors} (see Laslier \& Laslier~\cite{LL}).

Many particular examples have been studied in the literature, mostly in the two-colour case:
see, for example, Chen \& Wei~\cite{CW}, Chen \& Kuba~\cite{CK} and Kuba, Mahmoud \& Panholzer~\cite{KMP}
and Aguech, Lasmar \& Selmi~\cite{ALS}.

In recent work, Kuba \& Mahmoud~\cite{KM} and Kuba \& Sulzbach~\cite{KS}
considered the general case of two-colour P{\'o}lya urns with multiple drawing 
and proved third-order asymptotics of the composition of the urn under two hypothesis. 
The standard {\it balance} hypothesis makes the total number in the urn deterministic. 
The {\it affinity} hypothesis is the assumption that, 
if one denotes by $W_n$ the number of white balls in the urn at time~$n$,
then there exists two deterministic sequences $(\alpha_n)_{n\geq 0}$ and $(\beta_n)_{n\geq 0}$ such that, 
for all integers $n$, $\mathbb E[W_{n+1}|\mathcal F_n] = \alpha_n W_n +\beta_n$.

\subsection{The main contributions of this paper}
The main idea of this paper is to apply the stochastic approximation methods to the multi-drawing problem.
Our main motivation for doing this was the successful application 
by Laruelle \& Pag\`es~\cite{LP} of these techniques to the
classical P\'olya urn model with random replacement matrices.

Using these methods we prove results in the $d$-colour ($d\geq 2$), multi-drawing case.
In the case $d\geq 3$, these are the first general results in the literature.
In the case $d\geq 2$, our results are analogous to Kuba \& Mahmoud~\cite{KM}, 
but importantly do not require the somewhat artificial affinity hypothesis.

Furthermore, we prove partial results on the two-colour {\it non-balanced} case; 
these are the first results in the literature for this situation,
apart from the study of particular examples in, for example,~\cite{ALS}.

%which allows us to remove the affinity hypothesis from the results of Kuba \& Mahmoud~\cite{KMa, KMb},
%and also partially extend it to $d$-colour urns.
%Our main contribution is to get rid of this second hypothesis by using stochastic approximation methods.
%Also, we do not restrict ourselves to two-colour P{\'o}lya urns since our methods also apply for multi-colour urns.
%Under these less restrictive assumptions, we are able to prove up to second-order asymptotics for the urn composition.

\subsection{Definition of the model and main assumptions}

%%%% Precise definition of our model
The model we study is defined as follows. An urn contains balls of
$d\geq 2$ different colours. At time 0, the urn contains
$U_{0,i}\geq 0$ balls of colour $i$, for all $1\leq i\leq d$. We
assume that the urn is originally non empty, namely $\sum_{i=1}^d U_{0,i}>0$. 
We fix an integer $m\geq 1$ and a
replacement rule $R\,:\,\Sigma^{(d)}_m \to \mathbb Z^d$, where
\[\Sigma^{(d)}_m = \{(v_1, \ldots, v_d)\in \mathbb N^d \,:\, \sum_{i=1}^d v_i = m\}.\]
We denote by $R_1, \ldots, R_d$ the coefficient functions of $R$. At each
(discrete) time step~$n\geq 1$, 
we draw $m$ balls in the urn {\it with or without
replacement}, and denote by $\xi_{n,i}$ the number of balls of
colour $i$ among those $m$ balls. Let $\xi_n = (\xi_{n,1}, \ldots,
\xi_{n,d})$. Conditionally on $\xi_n = v$, we then add into the urn
$R_i(v)$ balls of type $i$ into the urn, for all $1\leq i\leq d$.
More precisely, we have $U_n = U_{n-1} + R(\xi_n)$, for all $n\geq 1$, 
where $U_{n,i}$ is the number of balls of colour $i$ in the urn at time $n$, 
and $U_n = (U_{n,1}, \ldots, U_{n,d})$.

Let us denote by $T_n$ the total number of balls
and $Z_{n,i} = \nicefrac{U_{n,i}}{T_n}$ the proportion of
balls of type $i$ in the urn at time~$n$.
Note that, with these notations, in the with-replacement case,
\[P_n(v):= \mathbb P(\xi_{n+1} = v | \mathcal F_n) = \binom m {v_1, \ldots, v_d} \prod_{i=1}^d Z_{n,i}^{v_i},\]
while in the without-replacement case,
\[P_n(v):= \mathbb P(\xi_{n+1} = v | \mathcal F_n) = \frac{1}{\binom{T_n}{m}}\, \prod_{i=1}^d \binom{U_{n,i}}{v_i}.\]

In the literature, it is often assumed that the urn is {\it balanced}, 
meaning that, for all $v\in\Sigma_m^{(d)}$, $r(v):=\sum_{i=1}^d R_i(v) = S$, where $S$ is a positive integer. 
Without this assumption, one needs to control the speed of convergence of $T_n/n$ to its limit, 
and this is not yet understood for $m\geq 2$ (i.e. in the multiple-drawing case). 
It is also standard to assume that the urn is {\it tenable}, 
meaning that it is never asked to remove from the urn balls that are not in the urn: 
we give necessary and sufficient conditions for the tenability to be achieved. 
We give labels to these two assumptions since they will be used all along the article:
\begin{itemize}
\item[\texttt{(B)}] For all $v\in\Sigma_m^{(d)}$, $r(v) := \sum_{i=1}^d R_i(v) = S$. 
\item[\texttt{(T)}] The urn scheme is tenable. 
\end{itemize}
Although the tenability assumption is natural (an alternative would be to work conditionally on the event that no impossible configuration happens), we believe it is an interesting and challenging open question to remove the balance hypothesis altogether. Note that in the classical $m=1$ case, strong results can be proved without the balance assumption (see Janson~\cite{Janson04}); the reason is that the urn scheme embedded in continuous time is a multi-type Galton-Watson process, which is no longer the case when $m\geq 2$. Although our main results for $d$-colour urns require the balance assumption, we are able to prove partial results for two-colour non-balanced urns.

The following function from $\Sigma^{\scriptscriptstyle (d)} = \{(x_1, \ldots, x_d)\in [0,1]^d \colon \sum_{i=1}^d x_i = 1\}$ onto 
$\{(x_1, \ldots, x_d)\in \mathbb R^d \colon \sum_{i=1}^d x_i = 0\}$ (two $(d-1)$-dimensional spaces) and its
zeros will play a crucial role in the article (recall that $r(v):= \sum_{i=1}^d R_i(v)$):
\[h(x) = \sum_{v\in\Sigma_m^{(d)}} \binom {m}{v_1, \ldots, v_d} \left(\prod_{i=1}^d x_i^{v_i}\right)\, \Big(R(v) - r(v)\, x\Big).\]
We will especially focus on the stable zeros of $h$, i.e.\ the zeros at which 
all eigenvalues of $\nabla h$ (the Jacobian matrix of~$h$) have a negative real part: 
we let $\mathcal Z(h)$ denote the set of zeros of $h$, and $\mathtt{Stable}(h)$ denote the set of stable zeros of $h$.

In this article, we focus on the renormalised composition vector $Z_n:= (Z_{n,1}, \ldots, Z_{n,d})$.
In view of the definition above, the process $(Z_n)_{n\geq 0}$ is a Markov process that depends on two parameters, 
the initial composition $Z_0$ and the replacement function $R$.
Before stating them in full detail in the rest of this introduction, 
let us summarise our main results.
We prove the following results for balanced urns:
\begin{itemize}
\item The {\it limit set} (see Definition~\ref{definition:limit_set}) of the renormalised composition vector $Z_n$ is almost surely a compact connected set of $\Sigma^{\scriptscriptstyle (d)}$ stable by flow of the differential equation $\dot{x} = h(x)$ (see Theorem~\ref{th:main}$(a)$). Note that, given a function~$h$, it is a non-trivial question to determine the compact connected sets stable by the flow of $\dot{x} = h(x)$ (we give some examples in Section~\ref{sec:examples}). Also, the limit set of~$Z_n$ a priori depends on the initial composition vector~$Z_0$. Reducing the number of colours to $d=2$ makes the matter much simpler and, in this particular case, we state almost sure convergence to a constant vector $\theta$ as long as the function~$h$ is not constant equal to zero (see Corollary~\ref{th:2col}$(a)$).
\item Assuming that $Z_n$ converges almost surely to a stable zero of $h$, 
we prove convergence of the fluctuations around this almost sure limit 
(see Theorem~\ref{th:main}$(b)$).
\item When $h\equiv 0$ (we call this case the {\it diagonal case}), 
we prove almost sure convergence of the renormalised composition vector to a random vector $Z_{\infty}$ (see Theorem~\ref{th:diag}).
\end{itemize}
In the two-colour non-balanced case, we prove almost sure convergence to a zero of $h$ (as long as $h\not\equiv 0$), 
and some partial result for the fluctuations around this almost sure limit (see Theorem~\ref{th:non-balanced}).

\subsection{Main result for $d$-colour balanced urns}

\begin{definition}[{See, e.g., Pemantle~\cite[Definition~2.11]{Pemantle}}]\label{definition:limit_set}
Given a stochastic process $(\Pi_n)_{n\geq 0}$, we define its limit set as
\[L(\Pi) := \bigcap_{n\geq 0} \overline{\bigcup_{m\geq n} \Pi_m}.\]
\end{definition}

We let $\mathcal Z(h)$ denote the set of zeros of $h$: 
$\mathcal Z(h) = \{z\in\Sigma^{\sss (d)}\colon h(z) = 0\}$.
We also recall the definition of an attractor (see, e.g.\ \cite[p.14]{Pemantle}):

\begin{definition}
For all $x\in\Sigma^{\sss (d)}$ and $t\geq 0$, we let $\Phi_t(x)$ be the value at time $t$ of the (unique because $h$ is Lipschitz) solution of the ODE $\dot y = h(y)$ started at $y(0) = x$.
An attractor of the ODE $\dot y = h(y)$ is a set $\mathcal A\in\Sigma^{\sss (d)}$ that admits a neighbourhood $\mathcal U$ such that
\[\omega(\mathcal U) := \bigcap_{t\geq 0}\overline{\bigcup_{s>t} \Phi_s(\mathcal U)} = \mathcal A.\]
\end{definition}

We define Assumptions (A1) and (A2) as:
\begin{itemize}
\item[(A1)] The zeros of $h$ on $\Sigma^{\sss (d)}$ are isolated and there exists $\theta\in\mathcal Z(h)$ such that, for all $z\in \Sigma^{\sss (d)}\setminus\mathcal Z(h)$, $\langle z-\theta, h(z)\rangle<0$.
\item[(A2)] 
\begin{itemize} 
\item[(i)] Almost surely for all $z\in\mathcal Z(h)\setminus\{\theta\}$, $\langle u-z, h(u)\rangle\geq 0$ for all $u\neq z$ in a neighbourhood of $z$ in $\Sigma^{\sss (d)}$.
\item[(ii)] For all $n\geq 0$, $Z_n \notin \mathcal Z(h)\setminus\{\theta\}$. 
\item[(iii)] In the without-replacement case, and if $\mathcal Z(h)\setminus \{\theta\}\neq \varnothing$, for all $1\leq i\leq d$, $U_{n,i} \to+\infty$ almost surely when $n\to+\infty$.
\end{itemize}
\end{itemize}

{\bf NB:} If (A1) hold and $\mathcal Z(h) = \{\theta\}$, then (A2) holds.

\begin{thm}[Balanced and tenable $d$-colour urns]\label{th:main}
Under Assumptions~\texttt{(B)} and~\texttt{(T)}, we have:
\begin{enumerate}[(a)]
\item The limit set of the renormalised composition vector $Z_n= (Z_{n,1}, \ldots, Z_{n,d})$ 
is almost surely a compact connected set of $\Sigma^{\scriptscriptstyle (d)}$ 
stable by the flow of the differential equation $\dot{x} = h(x)$, 
and the flow of the ODE restricted to $L(Z)$ admits no other attractor except $L(Z)$ itself.
If (A1) holds, then $L(Z) = \{x\}$ for some $x\in\mathcal Z(h)$.
If, in addition, (A2) holds, then $Z_n$ converges almost surely to $\theta$.

\item Assume that there exists a stable zero $\theta$ of $h$ 
such that $Z_n$ converges almost surely to $\theta$ when $n$ goes to infinity.
Let $\Lambda$ be the eigenvalue of $-\nabla h(\theta)$ with the smallest real part,
let \[\Gamma=\frac1{S^2}\,\sum_{v\in\Sigma_m^{(d)}} \binom m{v_1, \ldots, v_d} \left(\prod_{i=1}^d \theta_i^{v_i}\right) (R(v)-S\theta)(R(v)-S\theta)^t.\]
Then,
\begin{itemize}
\item if $\mathtt{Re}(\Lambda)>\nicefrac{S}2$, then $\sqrt{n}(Z_n-\theta) \to \mathcal N(0, \Sigma)$,
in distribution when $n\to\infty$, where
\[\Sigma = \int_{0}^{+\infty}\Bigg(\exp\bigg(\Big(\frac{\nabla h(\theta)}{S}+\frac{\mathtt{Id}}{2}\Big)u\bigg)\Bigg)^{\!\!t}\Gamma\,
\exp\bigg(\Big(\frac{\nabla h(\theta)}{S}+\frac{\mathtt{Id}}{2}\Big)u\bigg)\, du.\]
\end{itemize}
Assume additionally that all Jordan blocks of $\nabla h(\theta)$ associated to $\Lambda$ are of size $1$. Then,
\begin{itemize}
\item if $\mathtt{Re}(\Lambda)=\nicefrac{S}2$, 
then $\sqrt{\nicefrac{n}{\log n}}(Z_n-\theta) \to \mathcal N(0, \Sigma)$,
in distribution when $n\to\infty$, where
\[\Sigma = \lim_{n\to\infty}\frac1{\log n}
\int_{0}^{\log n}\Bigg(\exp\bigg(\Big(\frac{\nabla h(\theta)}{S}+\frac{\mathtt{Id}}{2}\Big)u\bigg)\Bigg)^{\!\!t}\Gamma\,
\exp\bigg(\Big(\frac{\nabla h(\theta)}{S}+\frac{\mathtt{Id}}{2}\Big)u\bigg)\, du;\]
\item if $\mathtt{Re}(\Lambda)<\nicefrac{S}2$, then $n^{\nicefrac{\mathtt{Re}(\Lambda)}{S}}(Z_n-\theta)$ converges almost surely to a random variable.
\end{itemize}
\end{enumerate}
\end{thm}

{\bf Remark:} 
The fact that both $Z_n$ and $\theta$ are in $\Sigma^{\sss (d)}$ implies that
the matrix $\Sigma$ satisfies $\Sigma \cdot(1, \ldots, 1)^t = 0$.

{\bf Remark:} When applying this theorem to particular examples, 
the difficulty is to understand the flow of the differential equation $\dot{x} = h(x)$.
It is in general non-trivial to describe the limit set of $(Z_n)_{n\geq 0}$ 
(see, e.g., Laslier \& Laslier~\cite{LL} where this is carried out for one particular example), 
and this limit set is a priori not a point and depends on the initial composition $Z_0$.
The strength of our result is that, 
if one can prove, for a particular replacement rule and for one particular initial composition vector $Z_0$, 
that the renormalised composition vector $Z_n$ converges almost surely to a stable zero of $h$, 
then the fluctuations are given relatively easily by Theorem~\ref{th:main}$(b)$. 
We give numerous examples in Section~\ref{sec:examples}.

To discuss our main result, let us first apply it to the well-known particular case $m=1$.
In that case, the replacement rule is encoded by a replacement matrix $R=(R_{i,j})_{1\leq i,j\leq d}$.
At every time step, we pick a ball uniformly at random in the urn, denote by $i$ its colour, 
and put it back in the urn together with $R_{i,j}$ balls of colour $j$, for all $1\leq i,j\leq d$.
Note that this translates into our framework as $R_j(\boldsymbol e_i) = R_{i,j}$ for all $1\leq i,j\leq d$ 
(where $\boldsymbol e_i$ is the vector whose coordinates are all equal to~0 except the $i$th, which is equal to~1).
We thus have $h(x) = (R-S \mathtt{Id}_d) x$. 
If we assume that $R$ is irreducible and that the urn is tenable, 
the Perron-Frobenius theorem implies that $S$ is the spectral radius of $R$, 
the multiplicity of $S$ as eigenvalue of $R$ is one, 
and there exists an eigenvector $\varpi$ associated to $S$ 
whose coordinates are all non-negative and such that $\sum_{i=1}^d \varpi_i = 1$.

In particular, $\varpi$ is the unique zero of $h$ on $\Sigma^{\scriptscriptstyle (d)}$, 
and the eigenvalues of $R-S\mathtt{Id}_d$ restricted to $\Sigma^{\scriptscriptstyle (d)}$ 
all have negative real parts.
Thus $\varpi$ is a stable zero of $h$, and, for all $x\neq \varpi$, 
$\langle h(x),x-\varpi\rangle = \langle R-S\mathtt{Id}_d)(x-\varpi), x-\varpi \rangle < 0$.
Theorem~\ref{th:main}$(a)$ applies and gives almost sure convergence of $Z_n$ to $\varpi$ and
Theorem~\ref{th:main}$(b)$ applies straightforwardly, 
reproving some standard results from the literature 
(see Janson~\cite{Janson04}).

By analogy to the classical $m=1$ case, 
we can say that the cases for which $Z_n$ does not converge to a constant vector 
or converges to a constant vector that depends on $Z_0$ are the {\it non-irreducible} cases.
Unfortunately, we are not yet able to give a nice characterisation of those cases in terms of the replacement function $R$.
Theorem~\ref{th:main}$(a)$ gives a sufficient condition; 
namely, if there exists a zero $\theta$ of $h$ such that $\langle h(x), x-\theta\rangle<0$ 
for all $x\in\Sigma^{(d)}$, then $Z_n$ converges almost surely to $\theta$ 
(for all choices of $Z_0$ satisfying $\mathtt{(T)}$).
Note that, in the classical setting $m=1$, the non-irreducible cases are also more intricate; 
see Janson~\cite{Janson06} for two-colour triangular urns 
and Bose, Dasgupta \& Maulik~\cite{BDM} for the $d$-colour balanced and triangular case.

In the case when $h\equiv 0$, 
which we call {\it diagonal} case
(although~\cite{KM,KS} call it {\it triangular} 
- we choose to change the terminology so that it is coherent with the $m=1$ case), 
we prove the following:
\begin{thm}[Diagonal balanced case]\label{th:diag}
Under Assumption $\mathtt{(B)}$ and $\mathtt{(T)}$, if $h\equiv 0$, 
then $Z_n$ converges almost surely to a random vector $Z_{\infty}$ when $n$ tends to infinity.
\end{thm}

\subsection{Main results for the two-colour case}
In the two-colour case, the renormalised composition vector $Z_n$ of the urn at time $n$ is characterised 
by its first coordinate since $Z_{n,1}+Z_{n,2}=1$. 
Thus, it is enough to study the first coordinate and the problem becomes unidimensional and easier.
Let us introduce simplified notations for this special case.

In the two-colour case, we say that the urn contains white balls and black balls.
We denote by $W_n$ and $X_n = \nicefrac{W_n}{T_n}$ respectively the number and the proportion of white balls in the urn at time $n$.
The replacement function can be seen as a replacement matrix
\[R = \begin{pmatrix}
a_0 & b_0 \\
a_1 & b_1 \\

\vdots & \vdots \\

a_m & b_m
\end{pmatrix},\]
where the $a_i$'s, and $b_i$'s are integers 
(such that $a_i+b_i=S$ for all $1\leq i\leq m$ under Assumption~$\mathtt{(B)}$) .
At each (discrete) time step~$n$, 
we draw $m$ balls in the urn {\it with}- or {\it without}-replacement,
and denote by $\zeta_n$ the number of white balls among those $m$.
Conditionally on $\zeta_n = k$,
we then add into the urn $a_{m-k}$ white balls and $b_{m-k}$ black balls 
(we keep the notations of~\cite{KM, KS}).

Note that, we these notations, in the {\it with}-replacement case,
\[\mathbb P(\zeta_{n+1} = k | \mathcal F_n) = \binom m k X_n^k (1-X_n)^{m-k};\]
and in the {\it without}-replacement case,
\[\mathbb P(\zeta_{n+1} = k | \mathcal F_n) = \binom m k \frac{(W_n)_k (T_n-W_n)_{m-k}}{(T_n)_m},\]
where $(n)_q = n (n-1)\cdots (n-q+1)$ for all integers $q$ and $n$.

The following result is a corollary of Theorem~\ref{th:main}.
\begin{cor}[Two-colour balanced urns]\label{th:2col}
Assume $d=2$, $\mathtt{(B)}$ and $\mathtt{(T)}$ and
let
\[g(x):= \sum_{k=0}^m \binom{m}{k} x^k (1-x)^{m-k} (a_{m-k}-S x).\]
\begin{enumerate}[(a)]
\item If $g\not\equiv 0$, then there exists a random variable $\theta_\star\in[0,1]$ such that
the proportion of white balls in the urn at time $n$, 
denoted by $X_n$, converges almost surely to $\theta_\star$. Furthermore, we have $g(\theta_\star)=0$ and $g'(\theta_\star)\leq 0$ almost surely.
\item Furthermore, conditionally on $\theta_\star$, the following holds with $\Lambda = -\nicefrac{g'(\theta_\star)}{S}$ and 
\[\Gamma = \frac1{S^2} \sum_{k=0}^m \binom mk \theta_\star^k (1-\theta_\star)^{m-k} (a_{m-k}-S\theta_\star)^2:\]
\begin{itemize}
\item If $\Lambda >\nicefrac 12$, then $\sqrt n(X_n - \theta_\star) \to \mathcal N\big(0,\frac{\Gamma}{2\Lambda-1}\big)$ in distribution when $n\to\infty$.
\item If $\Lambda = \nicefrac 12$, then $\sqrt{\nicefrac{n}{\log n}} (X_n-\theta_\star) \to \mathcal N(0,\Gamma)$ in distribution when $n\to\infty$.
\item If $\Lambda <\nicefrac 12$, then $n^{\Lambda}(X_n - \theta_\star)$ converges almost surely to a finite random variable.
\end{itemize}
\end{enumerate}
\end{cor}

In the two--colour case, we are also able to get some partial results about the non-balanced case.
The following result is not a corollary of Theorem~\ref{th:main}, which only applies to balanced urn schemes.
Note that in the non-balance case, we need a stronger assumption than the tenability assumption: 
we are only able to study urns under the assumption that the total number of balls in the urn grows linearly in~$n$, 
i.e.\ $\liminf_{n\to\infty}\nicefrac{T_n}{n}>0$.
A way to ensure that this is true is for example to assume that $\min_{0\leq k\leq m} (a_k+b_k) \geq 1$.
\begin{thm}[Non-balanced two-colour case]\label{th:non-balanced}
Assume that $\liminf_{n\to\infty}\nicefrac{T_n}{n}>0$ and let
\[\tilde g(x):= \sum_{k=0}^m \binom m k x^k (1-x)^{m-k} \big((1-x)a_{m-k} - x b_{m-k}\big).\]
\begin{enumerate}[(a)]
\item If $\tilde g\not\equiv 0$,
then there exists a random variable $\theta_\star\in[0,1]$ such that the proportion of white balls in the urn, denoted by $X_n$, 
converges almost surely to $\theta_\star$. Furthermore, we have $\tilde g(\theta_\star) = 0$ and $\tilde g'(\theta_\star)\leq 0$ almost surely.
\item Conditionally on $\theta_{\star}$, let
\[\omega:= \sum_{k=0}^m \binom m k \theta_\star^k (1-\theta_\star)^{m-k} c_{m-k},\]
where $c_i := a_i+b_i$ for all $0\leq i \leq m$,
\[H(x) := \sum_{k=0}^m \binom m k x^k (1-x)^{m-k} \big((1-x)a_{m-k} - x b_{m-k}\big)^2,\]
and
\[\lambda := \frac{|h'(\theta_\star)|}{\omega}\quad \text{ and } \quad \sigma^2:= \frac{H(\theta_\star)}{\omega^2}.\]
Assume that $\sigma^2>0$. 
Then, if $\lambda > \nicefrac 12$, 
\[\sqrt n(Z_n-\theta_\star) \to \mathcal N\Big(0,\frac{\sigma^2}{2\lambda-1}\Big) \quad \text{in distribution when }n\to\infty.\]
\end{enumerate}
\end{thm}

{\bf Remark: } Note that in both Corollary~\ref{th:2col} and Theorem~\ref{th:non-balanced},
Statement~$(a)$ gives almost sure convergence to a random variable $\theta_\star$ 
which belongs to the set of zeros of the function $g$ (resp. $\tilde g$). 
The fact that~$g$ (resp. $\tilde g$) is a non-zero polynomial function ensures that 
this set is a set of at most $m$ isolated points of $[0,1]$. 
If~$g$ admits a unique zero $\theta_0$ in $[0,1]$ such that $g'(\theta_0)\leq 0$, 
then the proportion of white balls converges almost surely to this zero independently of the initial composition of the urn;
we give examples in Section~\ref{sec:examples}. 
This particular case is what we called the {\it irreducible} case when discussing Theorem~\ref{th:main}$(a)$.
If~$g$ admits at least two such zeros, then the almost sure limit of $X_n$ depends on the initial composition $X_0$;
we refer the reader to Example~4.1.5 for such an example. 

\subsection{Plan of the article}
The end of this introduction section (see Section~\ref{sec:tenability}) is devoted to 
stating  an algebraic necessary and sufficient condition for the urn to be tenable: 
it is a straightforward generalisation of the two-colour case studied by Kuba and Sulzbach~\cite[Lemma~1]{KS}.
Section~\ref{sec:LLN&CLT} contains the proofs of Theorems~\ref{th:main} and~\ref{th:diag}.
Section~\ref{sec:two_colour} treats the two-colour balanced case and contains the proof of Corollary~\ref{th:2col}.
We give in Section~\ref{sec:examples} numerous two-colour and three-colour examples and show how to apply our main results to them.
Finally, Section~\ref{sec:non_balanced} treats the two-colour non-balanced case and contains the proof of Theorem~\ref{th:non-balanced} as well as some examples.

\subsection{Tenability}\label{sec:tenability}

We here state a lemma giving necessary and sufficient conditions for tenability. This lemma and its proof are very similar to those stated in Kuba \& Sulzbach's article~\cite{KS} in the case of two-colour urns. The proof in the $d$-colour case would be very similar to the $d=2$ case and would not give additional insight. We therefore do not detail it.

\begin{lem}\label{lem:tenability}
Consider the urn process with initial composition $U_0$ and replacement function $R$.
For all $1\leq i\leq d$, we denote by $\nu_i$ the greatest common divisor of $\{R_i(v)\colon v\in \Sigma_m^{(d)}\setminus \{m\bs e_i\}\}$.
For all integers $\ell$, we denote by $[\ell]_i$ the remainder of division of $\ell$ by $\nu_i$.
\begin{itemize}
\item In the with-replacement case, the urn scheme is tenable if and only if, for all $1\leq i\leq d$, $R_i(v)\geq 0$ for all $v\neq m \bs e_i$,
and $R_i(m\bs e_i)\geq 0$ or $-R_i(m\bs e_i)$ is a divisor of $U_{0,i}$ and $\nu_i$.
\item In the without-replacement case, the urn scheme is tenable if and only if, for all $1\leq i\leq d$, $R_i(v)\geq -v_i$ for all $v\neq m\bs e_i$ and
\[R_i(m\bs e_i) \in [-m, \infty] \cup 
\left([-m-\nu_i+1, -m) \cap 
\Big\{\ell\in -\mathbb N \colon [U_{0,i}]_i \in \big\{[-\ell]_i, [-\ell+1]_i, \ldots, [m+\nu_i-1]_i\big\}\Big\}\right).\]
\end{itemize}
\end{lem}

\section{Law of large numbers and central limit theorem for balanced urn schemes}\label{sec:LLN&CLT}

\subsection{This model can be described as a stochastic algorithm}
\begin{lem}\label{lem:algo_sto}
Under Assumption~$\mathtt{(B)}$ and~$\mathtt{(T)}$,
the renormalised composition vector of the urn at time $n$, denoted by
$Z_n$, satisfies the following recursion:
\begin{equation}
Z_{n+1} = Z_n + \frac1{T_{n+1}} \big(h(Z_n) + \Delta M_{n+1} + \varepsilon_{n+1}\big),
\end{equation}
where
\[
h(x) = \sum_{v\in\Sigma_m^{(d)}} \binom {m}{v_1, \ldots, v_d} 
\left(\prod_{i=1}^d x_i^{v_i}\right) \big(R(v) - r(v)\, x\big)\quad
\text{ and }\quad
\Delta M_{n+1} =Y_{n+1}-\mathbb E[Y_{n+1}|\mathcal F_{n}]
\]
with \[Y_{n+1}=R(\xi_{n+1}) - r(\xi_{n+1}) Z_n,\]
where $\xi_{n+1}$ stands for the (random) vector of balls drawn at time $n+1$,
and finally $\varepsilon_{n+1} = 0$ in the \emph{with}-replacement case 
and $\varepsilon_{n+1}$ is a $\mathcal F_{n+1}$-adapted term satisfying $\varepsilon_n\to 0$ almost surely 
when $n$ tends to infinity in the \emph{without}-replacement case.
\end{lem}

\begin{proof}
Recall that 
$U_{n+1} = U_n + R(\xi_{n+1})$ and $T_{n+1}= T_n + r(\xi_{n+1})$,
implying that
\begin{align*}
Z_{n+1} - Z_n
&= \frac{U_n + R(\xi_{n+1})}{T_{n+1}} - \frac{U_n}{T_n} 
 = \frac1{T_{n+1}} \left(U_n + R(\xi_{n+1}) - \frac{T_n + r(\xi_{n+1})}{T_n}\, U_n\right)\\
&= \frac1{T_{n+1}} \left(R(\xi_{n+1}) - r(\xi_{n+1})\, Z_n\right)
 = \frac1{T_{n+1}} \big(\mathbb E[Y_{n+1}|\mathcal F_n] + \Delta M_{n+1}\big).
\end{align*}
Note that, in the {\it with}-replacement case,
\[\mathbb E[Y_{n+1}|\mathcal F_n]
= \sum_{v\in\Sigma_m^{(d)}} P_n(v) (R(v) - r(v)\, Z_n)
= h(Z_n),\]
since $P_n(v) = \binom{m}{v_1, \ldots, v_d} \prod_{i=1}^d Z_{n,i}^{v_i}$,
which concludes the proof in the {\it with}-replacement case.
In the {\it without}-replacement case,
\[\mathbb E[Y_{n+1}|\mathcal F_n]
= \sum_{v\in\Sigma_m^{(d)}} \frac{\prod_{i=1}^d \binom{U_{n,i}}{v_i}}{\binom{T_n}{m}}\, \big(R(v) - r(v)\, Z_n\big)
= h(Z_n) + \varepsilon_{n+1},\]
with
\begin{align}
\varepsilon_{n+1}
&:= \sum_{v\in\Sigma_m^{(d)}} \left(\frac{\prod_{i=1}^d \binom{U_{n,i}}{v_i}}{\binom{T_n}{m}} - \binom{m}{v_1, \ldots, v_d} \prod_{i=1}^d Z_{n,i}^{v_i}\right) \big(R(v)-r(v) Z_n\big)\notag\\
&= \sum_{v\in\Sigma_m^{(d)}} \binom{m}{v_1, \ldots, v_d} \big(R(v) - r(v) Z_n\big)
\left(\frac{\prod_{i=1}^d\prod_{j=0}^{v_i-1} (Z_{n,i}-\nicefrac{j}{T_n})}{\prod_{j=0}^{m-1}(1-\nicefrac{j}{T_n})}-\prod_{i=1}^d Z_{n,i}^{v_i}\right).\label{eq:def_eps}
\end{align}
To conclude the proof, we show that, almost surely as $n\to+\infty$,
\begin{equation}\label{eq:UB_eps}
\|\varepsilon_{n+1}\|_1 = \mathcal O(\nicefrac1{T_n}).
\end{equation}
This claim is obvious in the with-replacement case; we thus assume that we are in the without-replacement case.
Note that
\[\|\varepsilon_{n+1}\|_1
\leq 2S\sum_{v\in\Sigma_m^{(d)}} \binom m{v_1, \ldots, v_d} \prod_{i=1}^d Z_{n,i}^{v_i}
\left|\frac{\prod_{i=1}^d\prod_{j=0}^{v_i-1}(1-\nicefrac j{U_{n,i}})}{\prod_{i=0}^{m-1} (1-\nicefrac i{T_n})} - 1\right|,\]
and, for all $v\in \Sigma_m^{(d)}$,
\[\frac{\prod_{i=1}^d\prod_{j=0}^{v_i-1}(1-\nicefrac j{U_{n,i}})}{\prod_{i=0}^{m-1} (1-\nicefrac i{T_n})} - 1
\leq \frac{1}{(1-\frac{m-1}{T_n})^m} - 1 = \mathcal O(\nicefrac1{T_n}),\]
as $n\to+\infty$ (recall that, by Assumption \texttt{(B)}, $T_n = T_0 + nS\to+\infty$).
Moreover, for all $v\in \Sigma_m^{(d)}$ such that $v_i\leq U_{n,i}$ for all $1\leq i\leq d$,
\[1-\frac{\prod_{i=1}^d\prod_{j=0}^{v_i-1}(1-\nicefrac j{U_{n,i}})}{\prod_{i=1}^m (1-\nicefrac i{T_n})}
\leq \sum_{i=1}^d\sum_{j=0}^{v_i-1} \frac j{U_{n,i}}
\leq \sum_{i=1}^d \frac{v_i^2}{U_{n,i}} = \frac1{T_n} \sum_{i=1}^d \frac{v_i^2}{Z_{n,i}},\]
because, 
for all integers $a\geq 1$, for all $(x_1, \ldots, x_a)\in [0,1]^a$, 
$\prod_{i=1}^a (1-x_i)\geq 1-\sum_{i=1}^a x_i$
(this can be proved straightforwardly by induction on $a$).
For all integers $n$, we let $\Sigma(n)$ denote the subset of $\Sigma_m^{\sss (d)}$ such that
$v_i\leq U_{n,i}$ for all $1\leq i\leq d$; with this notation, we have, almost surely as $n\to+\infty$
\begin{align*}
&\sum_{v\in\Sigma(n)} 
\binom{m}{v_1, \ldots, v_d} \prod_{i=1}^d Z_{n,i}^{v_i}
\left|\frac{\prod_{i=1}^d\prod_{j=0}^{v_i-1}(1-\nicefrac j{U_{n,i}})}{\prod_{i=0}^{m-1} (1-\nicefrac i{T_n})} - 1\right|\\
&\hspace{2cm}\leq \mathcal O\Big(\frac1{T_n}\Big) \sum_{v\in\Sigma(n)} 
\binom m{v_1, \ldots, v_d} \prod_{i=1}^d Z_{n,i}^{v_i}
\Big(1+\sum_{i=1}^d \frac{v_i^2}{Z_{n,i}}\Big),\\
%&\hspace{2cm}=\sum_{v\in\Sigma_m^{(d)}} \binom m{v_1, \ldots, v_d} \prod_{i=1}^d Z_{n,i}^{v_i}
%\Big(1 + \sum_{i=1}^d \frac{v_i^2}{Z_{n,i}}\Big)\\
&\hspace{2cm}= \mathcal O\Big(\frac1{T_n}\Big)\left(1 +  \sum_{v\in\Sigma_m^{(d)}} \binom{m}{v_1, \ldots, v_d}\sum_{i=1}^d v_i^2 \prod_{k=1}^d Z_{n,k}^{v_k-\bs 1_{k=i}}\right)\\
&\hspace{2cm}\leq (1+ m^2 d^m)\mathcal O\Big(\frac1{T_n}\Big)  = \mathcal O\Big(\frac1{T_n}\Big).
\end{align*}
Indeed, if $v_i = 0$, then $v_i^2 \prod_{k=1}^d Z_{n,k}^{v_k-\bs 1_{k=i}} = 0$, otherwise, $v_i^2 \prod_{k=1}^d Z_{n,k}^{v_k-\bs 1_{k=i}}\leq v_i^2\leq m^2$.
For $v\in \Sigma_m^{(d)}$ such that there exists $1\leq i\leq d$ satisfying $v_i>U_{n,i}$, we have
\[1-\frac{\prod_{i=1}^d\prod_{j=0}^{v_i-1}(1-\nicefrac j{U_{n,i}})}{\prod_{i=1}^m (1-\nicefrac i{T_n})}=1,\]
and thus
\begin{align*}
&\sum_{v\in\Sigma(n)^c} 
\binom m{v_1, \ldots, v_d} \prod_{i=1}^d Z_{n,i}^{v_i}
\left|\frac{\prod_{i=1}^d\prod_{j=0}^{v_i-1}(1-\nicefrac j{U_{n,i}})}{\prod_{i=0}^{m-1} (1-\nicefrac i{T_n})} - 1\right|\\
&\hspace{2cm}= \sum_{v\in\Sigma(n)^c} 
\binom m{v_1, \ldots, v_d} \left(\prod_{i=1, v_i\leq U_{n,i}}^d Z_{n,i}^{v_i}\right)\left(\prod_{i=1, v_i> U_{n,i}}^d Z_{n,i}^{v_i}\right)\\
&\hspace{2cm}\leq \sum_{v\in\Sigma(n)^c}\binom m{v_1, \ldots, v_d} \prod_{i =1, v_i>U_{n,i}}^d \left(\frac{v_i\bs 1_{v_i>U_{n,i}}}{T_n}\right)^{\!\!v_i}
\leq \frac{(m d)^m}{T_n}.
\end{align*}
In total, we thus get that $\|\varepsilon_{n+1}\|_1 = \mathcal O(\nicefrac1{T_n}) = \mathcal O(\nicefrac 1n)$ almost surely as $n\to+\infty$, which concludes the proof.
\end{proof}

\subsection{Strong law of large numbers (Proof of Theorem~\ref{th:main}$(a)$)}
To prove Theorem~1$(a)$, we use the following result:
\begin{thm}[See, e.g., {\cite[Cor. 2.15]{Pemantle}}]\label{th:pemantle}
If $(Z_n)_{n\geq 0}$ is a sequence of random variables satisfying Equation~\eqref{eq:algo_sto} and if, in addition, $h$ is a Lipschitz function, $\sup_{n\geq 1}\|\Delta M_n\|<+\infty$, and $\sum_{n\geq 1}\|\varepsilon_n\|/n<+\infty$ almost surely, then the limit set $L(Z)$ is connected. Furthermore, the flow of $\dot y = h(y)$ restricted to $L(Z)$ has no other attractor except $L(Z)$ itself.
\end{thm}

We check that the assumptions of Theorem~\ref{th:pemantle} are satisfied in our case:
The function $h$ is Lipschitz and the noise is bounded: for all $i\geq 1$,
\[\|\Delta M_i\| \leq \|Y_i\| + \mathbb E[\|Y_i\||\mathcal F_{i-1}]
\leq 2\max_{v\in \Sigma_m^{\sss (d)}} \|R(v)\|,\]
by definition of $Y_i$.
Furthermore, by~\eqref{eq:UB_eps}, and by Assumption~\texttt{(B)}, $\|\varepsilon_n\|/n = \mathcal O(\nicefrac1{n^2})$ almost surely as $n\to+\infty$, which implies
$\sum_{n\geq 1}\|\varepsilon_n\|/n<+\infty$.
Therefore, Theorem~\ref{th:pemantle} implies that $L(Z)$ is connected
and that the  flow of the ODE $\dot y = h(y)$ restricted to $L(Z)$ 
has no attractor except $L(Z)$ itself. 
By Assumption (A1), $\theta$ is an attractor of the ODE. 
Therefore, if $\theta\in L(Z)$, then $L(Z) = \{\theta\}$.
We now assume that $\theta\notin L(Z)$, and show that this implies that $L(Z) = \{x\}$ for some $x\in\mathcal Z(h)\setminus\{\theta\}$.
Our argument is in 2 steps: 
\begin{itemize}
\item[(1)] We first use (A1) to prove that there exists $x$, a zero of $h$ that belongs to $L(Z)$: if $\theta$ is the only zero of $F$, this is a contradiction, which implies $L(Z) = \{\theta\}$ as claimed.
\item[(2)] In the case when $\theta$ is not the only zero of $h$, 
we use (A1) again to prove that $L(Z) =\{x\}$ for some $x\in \mathcal Z(h)\setminus\{\theta\}$.
\end{itemize}

(1) Because $L(Z)$ is closed by definition, assuming that $\theta\notin L(Z)$ implies that
there exists $x\in L(Z)$ such that $\min_{u\in L(Z)} \|u-\theta\| = \|x-\theta\|$.
If $x\notin \mathcal Z(h)$, then, by Assumption (A1), $\langle x-\theta, h(x)\rangle <0$. 
By definition of the flow, for all $t\geq 0$,
\[\partial_t\|\Phi_t(x)-\theta\|^2 = 2\langle \partial_t \Phi_t(x), \Phi_t(x)-\theta\rangle = 2\langle h(\Phi_t(x)), \Phi_t(x)-\theta\rangle,\]
and in particular, at $t=0$,
\[\partial_t\|\Phi_t(x)-\theta\|^2_{|t=0} = 2\langle h(x), x-\theta\rangle <0.\]
Therefore, for $t$ small enough, $\|\Phi_t(x)-\theta\|<\|x-\theta\|$ and thus $\Phi_t(x)\notin L(Z)$, 
which is impossible because $L(Z)$ is invariant by the flow of $\dot y = h(y)$.
Therefore, $x\in \mathcal Z(h)$. 
Note that if $\mathcal Z(h) = \{\theta\}$, this is impossible (since, by assumption $\theta\notin L(Z)$), and we can conclude by contradiction that $\theta\in L(Z)$, and thus $L(Z) = \{\theta\}$. (In that case, skip Step (2) below.)

(2) If $\mathcal Z(h)\neq\{\theta\}$, we have shown that, if $\theta\notin L(Z)$, then there exists $x\neq \theta$ a zero of $h$ that belongs to $L(Z)$, and $\|x-\theta\| = \min_{u\in L(Z)}\|u-\theta\|$. We now aim to show that $L(Z) = \{x\}$: to do, so we reason by contradiction and assume that $L(Z)\neq \{x\}$.
Since the zeros of $h$ are isolated by assumption, we can choose $\varepsilon>0$ such that 
\[\mathcal U_\varepsilon:=L(Z)\cap\bar{\mathcal B}(\theta, (1+2\varepsilon)\|x-\theta\|)\] 
contains no zero of~$h$ except~$x$ itself ($\bar{\mathcal B}(z,r)$ denotes the closed ball centred at $z$ and of radius $r$). See Figure~\ref{fig} for a representation of the different sets used in this part of the proof.
By choosing $\varepsilon$ small enough, we can also make sure that $L(Z)$ is not included in
\[\mathcal A_{\varepsilon}:= L(Z)\cap \bar{\mathcal B}(\theta, (1+\varepsilon)\|x-\theta\|).\] 
\begin{figure}
\begin{center}\includegraphics[width=5cm]{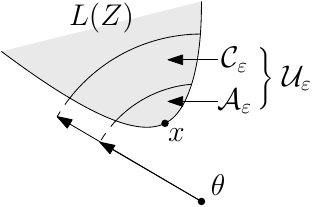}\end{center}
\caption{The sets $\mathcal A_{\varepsilon}, \mathcal C_{\varepsilon}$, and $\mathcal U_{\varepsilon}$ as used in the proof of Theorem 1.1$(a)$. The radius of the small circle is $(1+\varepsilon)\|x-\theta\|$, the radius of the large circle is $(1+2\varepsilon)\|x-\theta\|$.}
\label{fig}
\end{figure}
We let 
\[\mathcal C_{\varepsilon}=L(Z)\cap \big(\bar{\mathcal B}(\theta, (1+2\varepsilon)\|x-\theta\|)\setminus\mathcal B(\theta, (1+\varepsilon)\|x-\theta\|)\big),\]
where ${\mathcal B}(z,r)$ denotes the open ball centred at $z$ and of radius $r$.
The set $\mathcal C_{\varepsilon}$ is by definition closed (and bounded), and for all $u\in\mathcal C_{\varepsilon}$, by Assumption (A1), $\langle h(u), u-\theta\rangle <0$.
By continuity of $h$, this implies that
\[\kappa:=\sup_{u\in \mathcal C_{\varepsilon}}\langle h(u), u-\theta\rangle <0.\]
From this, we deduce that $\mathcal A_\varepsilon$ is an attractor for the ODE $\dot y = h(y)$. 
Indeed, for all $u\in\mathcal C_{\varepsilon}$, 
for all $t\geq 0$, if $\Phi_t(u)\in \mathcal C_\varepsilon$,
\[\partial_t\|\Phi_t(u)-\theta\|=2\langle h(\Phi_t(u)), u-\theta\rangle\leq \kappa<0,\]
And thus for all $t$ large enough, $\Phi_t(u)\in \mathcal A_{\varepsilon}$ (recall that $\Phi_t(u)\in L(Z)$ because $L(Z)$ is stable by the flow, by Theorem~\ref{th:pemantle}).
Similarly, $\partial_t\|\Phi_t(u)-\theta\|\leq 0$ as long as $\Phi_t(u)\in\mathcal A_{\varepsilon}$.
Therefore, if $u\in\mathcal A_{\varepsilon}$, then the solution started at $u$ stays in $\mathcal A_{\varepsilon}$ for all times.
Since, by Theorem~1, 
the flow of the ODE restricted to $L(Z)$ admits no other attractor except $L(Z)$ itself, 
this means that
\[L(Z) = \mathcal A_\varepsilon,\]
which is impossible since we have chosen $\varepsilon$ such that $L(Z)$ is not included in $\mathcal A_\varepsilon$. 

We have thus shown that, under (A1), $Z_n\to x$ almost surely for some $x\in\mathcal Z(h)$. If $\mathcal Z(h) = \{\theta\}$, we can conclude that $Z_n \to \theta$. Otherwise, it remains to show that, under (A1-2), $Z_n\to\theta$ almost surely as $n\to+\infty$.

We reason by contradiction and assume that $Z_n \to x\in\mathcal Z(h)\setminus\{\theta\}$. 
By Equation~\eqref{eq:algo_sto}, we get that, for all $n\geq n_0$,
\[\|Z_{n+1}-x\|^2
= \|Z_n-x + \gamma_n(h(Z_n)+\Delta M_{n+1}+ \varepsilon_{n+1})\|^2
\geq \|Z_n-x\|^2 + 2\gamma_n \langle Z_n-x, h(Z_n) + \Delta M_{n+1}+ \varepsilon_{n+1}\rangle,
\]
where we have set $\gamma_n = 1/T_{n+1}$ for all $n\geq 0$.
This implies
\begin{equation}\label{eq:norm-decr}
\mathbb E_n\|Z_{n+1}-x\|^2
\geq \mathbb E\|Z_n-x\|^2 + 2\gamma_n \langle Z_n-x, h(Z_n) \varepsilon_{n+1}\rangle,
\end{equation}
because $\Delta M_{n+1}$ is a martingale increment.
In the with-replacement case, $\varepsilon_{n+1} = 0$ and, by Assumption (A2-i), almost surely,
there exists~$n_0$ such that, for all $n\geq n_0$, $\langle h(Z_n), Z_n-z\rangle\geq 0$.
Therefore, for all $n\geq n_0$,
\[\mathbb E\|Z_{n+1}-x\|^2\geq \mathbb E\|Z_n-x\|^2\geq \mathbb E\|Z_{n_0}-x\|^2>0,\]
because by Assumption (A2-ii), $Z_n \neq x$ for all $n\geq 0$.
This is a contradiction, and we thus get that $Z_n\to \theta$ as claimed.

In the without-replacement case, 
we show that, on the event $Z_n \to x$, almost surely as $n\to+\infty$,
\begin{equation}\label{eq:eps_negl}
\langle Z_n-x, \varepsilon_{n+1}\rangle  = o\big(\langle Z_n-x, h(Z_n)\rangle\big).
\end{equation}
Note that, on the event that $Z_n\to x$, for all $n$ large enough, $Z_n$ is in a neighbourhood of $x$ that contains no other zero of the function $h$ (since by Assumption (A1), the zeros of $h$ are isolated). Also, by Assumption (A2-ii), $Z_n\neq x$ for all $n\geq 0$. Thus, by Assumption (A2-i), for all $n$ large enough, $\langle Z_n-x, h(Z_n)\rangle$ is non-null for all $n$ large enough, and~\eqref{eq:eps_negl} makes sense.
Before proving Equation~\eqref{eq:eps_negl}, 
we show how it allows us to conclude the proof of Theorem 1$(a)$:
Equation~\eqref{eq:eps_negl} implies 
that for all $\varepsilon\in(0,1)$, there exists $n_1\geq n_0$ such that, for all $n\geq n_1$,
\[\langle Z_n-x, \varepsilon_{n+1}\rangle\geq -\varepsilon\langle Z_n-x, h(Z_n)\rangle,\]
which, together with Equation~\eqref{eq:norm-decr}, gives
\[\mathbb E_n\|Z_{n+1}-x\|^2
\geq\mathbb E\|Z_n-x\|^2 + (1-\varepsilon)\gamma_n\langle Z_n-x, h(Z_n)\rangle
\geq \mathbb E\|Z_n-x\|^2,
\]
Therefore,
$\mathbb E\|Z_{n_1}-x\|>0$ implies that $\liminf\mathbb E\|Z_n-x\|>0$ for all $n\geq n_1$, 
which is impossible since we have assumed that $Z_n\to x$.
Therefore, $\mathbb E\|Z_{n_1}-x\|=0$, and thus $Z_{n_1} = x$ almost surely,
which contradicts the fact that, by Assumption (A2-ii), 
$Z_n\notin \mathcal Z(h)\setminus\{\theta\}$ almost surely for all $n\geq 0$.

To conclude the proof, it thus only remains to prove Equation~\eqref{eq:eps_negl} in the without-replacement case.
We use Equation~\eqref{eq:def_eps}:
on the event that $Z_n \to x\in\mathcal Z(h)$, we have
\begin{align*}
\varepsilon_{n+1}
&= \sum_{v\in\Sigma_m^{(d)}}\binom{m}{v_1, \ldots, v_d} \prod_{i=1}^d Z_{n,i}^{v_i} \big(R(v)-r(v) Z_n\big) 
\left(\frac{\prod_{i=1}^d (1-\nicefrac1{U_{n,i}}) \cdots (1-\nicefrac{(v_i-1)}{U_{n,i}})}{(1-\nicefrac1{T_n})\cdots(1-\nicefrac{(m-1)}{T_n})} - 1\right).
\end{align*}
As in the proof of Lemma~\ref{lem:algo_sto2}, we let $\Sigma(n)$ denote the subset of $\Sigma_m^{\sss (d)}$ such that $v_i\leq U_{n,i}$ for all $1\leq i\leq d$. 
Since, by Assumption (A2-iii), $U_{n,i}\to+\infty$ almost surely for all $1\leq i\leq d$, we have that $\Sigma(n) = \varnothing$ for all $n$ large enough.
Thus
\begin{align*}
\varepsilon_{n+1}&
=\sum_{v\in\Sigma(n)^c}
\binom{m}{v_1, \ldots, v_d} \prod_{i=1}^d Z_{n,i}^{v_i} \big(R(v)-r(v) Z_n\big) 
\left(\frac{\prod_{i=1}^d (1-\nicefrac1{U_{n,i}}) \cdots (1-\nicefrac{(v_i-1)}{U_{n,i}})}{(1-\nicefrac1{T_n})\cdots(1-\nicefrac{(m-1)}{T_n})} - 1\right)\\
&=\sum_{v\in\Sigma(n)^c}
\binom{m}{v_1, \ldots, v_d} \prod_{i=1}^d Z_{n,i}^{v_i} \big(R(v)-r(v) Z_n\big) 
\left(\mathcal O(\nicefrac1{T_n}) + \sum_{i=1}^d \mathcal O(\nicefrac1{U_{n,i}})\right),
\end{align*}
which implies
\[\langle Z_n-x, \varepsilon_{n+1}\rangle
= \left(\mathcal O(\nicefrac1{T_n}) + \sum_{i=1}^d \mathcal O(\nicefrac1{U_{n,i}})\right)
\langle Z_n -x, h(Z_n)\rangle= o\big(\langle Z_n -x, h(Z_n)\rangle\big),
\]
since, by Assumption (A2-iii), $U_{n,i}\to+\infty$ almost surely for all $1\leq i\leq d$.
This concludes the proof of~\eqref{eq:eps_negl} and thus of Theorem 1$(a)$.

\subsection{Central limit theorem (Proof of Theorem~\ref{th:main}$(b)$)}
To prove Theorem~\ref{th:main}$(b)$, the idea is again to apply standard theorems from the stochastic algorithms literature.
We state here a weak version of a result by Zhang~\cite{Zhang16}.
We also refer the reader to Laruelle \& Pag\`es~\cite[Appendix Theorem A.2]{LP}\footnote{There are different versions of this paper. We refer the reader to the ArXiV version \texttt{ArXiV:1101.2786}, which has been regularly updated by the authors.}.
\begin{thm}\label{CLT-Pages}
Assume that $\theta_n$ satisfies the recursion
\begin{equation}\label{eq:algosto}\forall n\geq n_0;\quad
\theta_{n+1}=\theta_n+\frac1{n+1}f(\theta_n)+\frac1{n+1}(\Delta \hat M_{n+1}+\hat \varepsilon_{n+1}),
\end{equation} where 
$f:\mathbb R^d \to \mathbb R^d$ is a differentiable non-null function, 
$\theta_0$ is a deterministic vector, 
for all $n\geq n_0$, $\Delta \hat M_n$ is an $\mathcal F_n$-increment martingale 
and $\hat \varepsilon_n$ is an $\mathcal F_n$-adapted remainder term.
Assume in addition that there exists $\theta\in \mathbb R^d$ a stable zero of $f$ such that $\theta_n$ converges to $\theta$ with positive probability.
Also assume that, for some $\delta>0$,
\[\sup_{n\geq 0}\mathbb E(\|\Delta \hat M_{n+1}\|^{2+\delta}|\mathcal F_n)< +\infty,\quad \text{and}\quad
\mathbb E(\Delta \hat M_{n+1}\Delta \hat M_{n+1}^t|\mathcal F_n)\stackrel{n\to \infty}{\longrightarrow}\hat \Gamma\quad \text{almost surely},\]
where $\hat \Gamma$ is a deterministic symmetric positive semi-definite matrix and
for some $\eta >0$
\begin{equation}\label{eq:32}
n^{\nicefrac32}\,\mathbb E\big[\|\hat \varepsilon_{n+1}\|^2\indi_{\|\theta_n-\theta\|\leq\eta} | \mathcal F_n\big] \stackrel{n\to \infty}{\longrightarrow} 0.
\end{equation}
Let $\hat \Lambda$ be the eigenvalue of $-\nabla f(\theta)$ with the largest real part. 
If we assume that $\theta_n$ converges almost surely to some deterministic limit $\theta$, then:
\begin{itemize}
\item If $\mathtt{Re}(\hat \Lambda)>\nicefrac{1}{2}$, then
$\sqrt{n}(\theta_n-\theta) \stackrel{n\to \infty}{\longrightarrow} \mathcal N(0,\hat \Sigma)$ in distribution,
where \[\hat \Sigma=\displaystyle\int_{0}^{+\infty}\Big(\mathtt e^{(\nabla f(\theta)+\frac{\mathtt{Id}}{2})u}\Big)^{\!\!t}\hat \Gamma\,
\mathtt e^{(\nabla f(\theta)+\frac{\mathtt{Id}}{2})u}\, du.\]
\end{itemize}
Assume additionally that $f$ is twice differentiable, and that all Jordan blocks of $\nabla f(\theta)$ associated to $\hat \Lambda$ have size~1. Then:
\begin{itemize}
\item If $\mathtt{Re}(\hat\Lambda)=\nicefrac{1}{2}$,
then $\displaystyle \sqrt{\frac{n}{\log n}}(\theta_n-\theta)\stackrel{n\to \infty}{\longrightarrow} \mathcal N(0,\hat \Sigma)$, in distribution,
where 
\[\hat \Sigma = \lim_{n\to\infty} \frac1{\log n} \int_0^{\log n} 
\Big(\mathtt e^{(\nabla f(\theta)+\frac{\mathtt{Id}}{2})u}\Big)^{\!\!t}\hat \Gamma\,
\mathtt e^{(\nabla f(\theta)+\frac{\mathtt{Id}}{2})u}\, du.\]
\item If $\mathtt{Re}(\hat \Lambda)<\nicefrac{1}{2}$,
then $n^{\mathtt{Re}(\hat \Lambda)}(\theta_n-\theta)$ converges almost surely to a finite random variable.
\end{itemize}
\end{thm}

{\bf Remark: }Note that Assumption~\eqref{eq:32} is not stated as such in~\cite{LP} and~\cite{Zhang16}, in which the assumption depends on the values of $\mathtt{Re}(\hat\Lambda)$. One can check that~\eqref{eq:32} is stronger than the assumptions of~\cite{LP} and~\cite{Zhang16}, and since it holds in our particular case, we only state this weaker version.

Applying this result to our framework gives the proof of the second claim of Theorem~\ref{th:main}:
\begin{proof}[Proof of Theorem~\ref{th:main}$(b)$]
Recall that (see Lemma~\ref{lem:algo_sto} for details) 
\[Z_{n+1}-Z_n = \frac1{T_{n+1}}\big(h(Z_n) + \Delta M_{n+1} + \varepsilon_{n+1}\big).\]
We also have $T_n = T_0 + n S$ (by Assumption~$\mathtt{(B)}$), which gives
\begin{align*}
Z_{n+1}-Z_n 
&= \frac1{(n+1)S} \big(h(Z_n) + \Delta M_{n+1} + \varepsilon_{n+1}\big) + 
\frac1{(n+1)S}\left(\frac{(n+1)S}{T_0+(n+1)S} - 1\right)\big(h(Z_n) + \Delta M_{n+1} + \varepsilon_{n+1}\big)\\
&= \frac1{n+1} \big(f(Z_n) + \Delta \hat M_{n+1} + \hat \varepsilon_{n+1}\big),
\end{align*}
where $f = \nicefrac h S$ (note that this function is infinitely differentiable), 
$\Delta \hat M_{n+1} = \nicefrac{\Delta M_{n+1}}{S}$ and
\[\hat \varepsilon_{n+1} 
= \frac{\varepsilon_{n+1}}{S} + \bigg(\frac{1}{1+\frac{T_0}{(n+1)S}}-1\bigg)\big(f(Z_n) + \Delta \hat M_{n+1} + \nicefrac{\varepsilon_{n+1}}{S}\big).\]
This last equality implies that, almost surely when $n$ tends to infinity,
\[\|\hat \varepsilon_{n+1}\| \leq \frac{\|\varepsilon_{n+1}\|}{S} + \mathcal O(\nicefrac1n).\]
Recall that in the {\it with}-replacement case, $\varepsilon_n = 0$ for all integers $n$. 
In the {\it without}-replacement case, we have already proved that $\|\varepsilon_n\| = \mathcal O(\nicefrac1{T_n}) = \mathcal O(\nicefrac1n)$
(see Equation~\eqref{eq:UB_eps}),
implying that %$r_{n+1} = \mathcal O(\nicefrac1n)$, and thus that 
$\|\hat \varepsilon_n\| = \mathcal O(\nicefrac1n)$,
and 
\begin{equation*}
n^{\nicefrac32}\,\mathbb E[\|\hat \varepsilon_n\|^2|\indi_{|\theta_n-\theta|\leq \eta}|\mathcal F_{n-1}]
= \mathcal O(n^{-\nicefrac12}) \to 0,
\end{equation*}
almost surely when $n$ tends to infinity, for all $\eta>0$.

We already mentioned that $\Delta M_{n+1}$ is almost surely bounded by $2(\|R\|+S)$, and thus $\sup_{n\geq 0}\mathbb E[\|\hat \Delta M_{n+1}\|^{2+\delta}|\mathcal F_n]<+\infty$ almost surely for all $\delta>0$.

Finally, note that
\begin{align*}
&\mathbb E[\Delta M_{n+1}\Delta M_{n+1}^t|\mathcal F_n]\\
&\hspace{2cm}= \mathbb E[Y_{n+1}Y_{n+1}^t|\mathcal F_n] - \mathbb E[Y_{n+1}|\mathcal F_n]\mathbb E[Y_{n+1}^t|\mathcal F_n]\\
&\hspace{2cm}= \sum_{v\in\Sigma_m^{(d)}} P_n(v) (R(v)-SZ_n)(R(v)-SZ_n)^t 
- \sum_{v,w\in\Sigma_m^{(d)}} P_n(v)P_n(w) (R(v)-SZ_n)(R(w)-SZ_n)^t. 
\end{align*}
Since $Z_n\to \theta$ almost surely, we have, in both the with- and the without-replacement cases that
$P_n(v) \to \binom m {v_1, \ldots, v_d} \prod_{i=1}^d \theta_i^{v_i}$ almost surely, implying that
\[\mathbb E[\Delta \hat M_{n+1}\Delta \hat M_{n+1}^t|\mathcal F_n]
= \frac1{S^2}\, \mathbb E[\Delta M_{n+1}\Delta M_{n+1}^t|\mathcal F_n] 
\stackrel{a.s.}{\longrightarrow} \Gamma,\]
where
\begin{align*}
S^2\,\Gamma 
=& \sum_{v\in\Sigma_m^{(d)}} \binom m{v_1, \ldots, v_d} \left(\prod_{i=1}^d \theta_i^{v_i}\right) (R(v)-S\theta)(R(v)-S\theta)^t\\
&- \sum_{v,w\in\Sigma_m^{(d)}} \binom m{v_1, \ldots, v_d}
\binom m{w_1, \ldots, w_d} \left(\prod_{i=1}^d \theta_i^{v_i+w_i} \right)
(R(v)-S\theta)(R(w)-S\theta)^t\\
=&  \sum_{v\in\Sigma_m^{(d)}} \binom m{v_1, \ldots, v_d} \left(\prod_{i=1}^d \theta_i^{v_i}\right) (R(v)-S\theta)(R(v)-S\theta)^t
- h(\theta) h(\theta)^t\\
=& \sum_{v\in\Sigma_m^{(d)}} \binom m{v_1, \ldots, v_d} \left(\prod_{i=1}^d \theta_i^{v_i}\right) (R(v)-S\theta)(R(v)-S\theta)^t,
\end{align*}
since $h(\theta)=0$. Note that $\Gamma$ is the limit of a sequence of symmetric positive, semi-definite matrices, and, as such, it is also symmetric positive and semi-definite. Therefore, Theorem~\ref{CLT-Pages} applies, which concludes the proof.
\end{proof}

{\bf Remark: }The result obtained in~\cite{Zhang16} also permits to treat the cases $\Re(\hat\Lambda) \leq \nicefrac S2$ when the Jordan blocks associated $\hat\Lambda$ are not all of size one. The asymptotic renormalisation then depends on the size of the largest blocks of $\nabla f(\theta)$ associated to $\hat\Lambda$. Applying this stronger version to our framework would give a generalisation of Theorem~\ref{th:main} to these cases. Since this generalisation is quite technical to state and since most of the examples fall under Theorem~\ref{th:main} as it is, we do not give more details.

\subsection{The diagonal case (proof of Theorem~\ref{th:diag})}\label{sec:diagonal}
Before proving Theorem~\ref{th:diag}, we give a simple characterisation of the diagonal balanced urn schemes:
\begin{lem}\label{lem:diagonal}
Under assumption~$\mathtt{(B)}$,
the function $h$ is identically null on $\Sigma^{(d)}$ if, and only if, 
there exists an integer $\sigma$ such that $S=m\sigma$ and
\begin{equation}\label{eq:truc}
R(v) = \sigma v \quad \text{ for all }\quad v\in\Sigma_m^{(d)}.
\end{equation}
\end{lem}

\begin{proof}
First note that, straightforwardly, the fact that there exists an integer $\sigma$ such that $S=m\sigma$ and such that Equation~\eqref{eq:truc} holds implies that $h\equiv 0$. 

The reverse implication is less straightforward.
First note that, for all $1\leq i\leq d$, 
$h(\bs e_i) = R(m\bs e_i) - S\bs e_i = 0$, 
implying that $R(m\bs e_i) = S\bs e_i$ for all $1\leq i\leq d$.

We reason by induction and prove that there exists $\sigma$ such that $R(v)=\sigma v$ for all vectors $v\in\Sigma_m^{(d)}$ having at most $k$ non-null coordinates. Note that if this result is true, it implies that $\sigma = \nicefrac S m$ and thus that $m$ divides $S$ (take $v = \bs e_1 + (m-1)\bs e_2$ and recall that $R_1(v) = \sigma$ is an integer).

Assume that this is true for some integer $k<m$, $R(v) = \sigma v$ for all vector $v\in\Sigma_m^{(d)}$ having at most $k$ non-null coordinates.
Let us prove that this property extends to vectors with at most $k+1$ non-zero coordinates. 
Without loss of generality, we can focus on vectors $v\in\Sigma_m^{(d)}$ such that the last $d-k-1$ coordinates are equal to zero, 
which is actually the set $\Sigma_m^{(k+1)}$ (with a slight abuse of notation since these vectors are still $d$-dimensional). 
Let $x = x_1 \bs e_1 + \ldots + x_{k+1} \bs e_{k+1}\in\Sigma^{(d)}$, we have, by assumption,
\begin{equation}\label{eq:induction}
\sum_{v\in\Sigma_m^{(k+1)}} \binom{m}{v_1, \ldots, v_{k+1}} \left(\prod_{i=1}^{k+1} x_i^{v_i}\right) R(v) = Sx.
\end{equation}
First note that for all $k+2\leq j\leq d$, we have
\[\sum_{v\in\Sigma_m^{(k+1)}} \binom{m}{v_1, \ldots, v_{k+1}} \left(\prod_{i=1}^{k+1} x_i^{v_i}\right) R_j(v) = 0.\]
Note that the tenability assumption implies that $R_j(v)\geq 0$ for all $v\notin \{m\bs e_1, \ldots, m\bs e_d\}$ (see Lemma~\ref{lem:tenability}).
Moreover, we have proved that for all $1\leq i\leq d$, $R(m\bs e_i) = S\bs e_i$, 
implying in particular that $R_j(m\bs e_i)\geq 0$ for all $1\leq j\leq d$. 
Therefore, the sum in the above display is a sum of non-negative terms. 
The fact that it is equal to zero thus implies that all its terms are null, 
and thus that $R_j(v) = 0~(= \sigma v_j)$, for all $k+2\leq j\leq d$. 
It remains to prove that the same equality, namely $R_j(v) = \sigma v_j$ holds for all $1\leq j\leq k+1$.

Note that
\[\sum_{v\in\Sigma_m^{(k+1)}} \binom{m}{v_1, \ldots, v_{k+1}} \left(\prod_{i=1}^{k+1} x_i^{v_i}\right) \big(R(v)-\sigma v\big) 
= \sum_{v\in\Sigma_m^{(k+1)}} \binom{m}{v_1, \ldots, v_{k+1}} \left(\prod_{i=1}^{k+1} x_i^{v_i}\indi_{v_i\neq 0}\right) 
\big(R(v)-\sigma v\big),\]
since all $v\in \Sigma_m^{(k+1)}$ having at least one null coordinate satisfy $R(v) = \sigma v$ by the induction hypothesis.
Therefore, Equation~\eqref{eq:induction} implies that
\begin{equation}\label{eq:Sx}
Sx 
= \sum_{v\in\Sigma_m^{(k+1)}} \binom{m}{v_1, \ldots, v_{k+1}} 
\left(\prod_{i=1}^{k+1} x_i^{v_i}\,\indi_{v_i\neq 0}\right) 
\left(R(v)-\sigma v\right) 
+ \sigma \sum_{v\in\Sigma_m^{(k+1)}} \binom{m}{v_1, \ldots, v_{k+1}} \left(\prod_{i=1}^{k+1} x_i^{v_i}\right) v.
\end{equation}
Note that, for all $1\leq j\leq k+1$,
\[x_j \frac{\partial}{\partial x_j} \left(x_1+\cdots+x_{k+1}\right)^m
=x_j \frac{\partial}{\partial x_j} 
\left(\sum_{v\in\Sigma_m^{(k+1)}} \binom{m}{v_1, \ldots, v_{k+1}} \left(\prod_{i=1}^{k+1} x_i^{v_i}\right)\right)
= \sum_{v\in\Sigma_m^{(k+1)}} \binom{m}{v_1, \ldots, v_{k+1}} \left(\prod_{i=1}^{k+1} x_i^{v_i}\right) v_j,\]
implying that
\[m x = \sum_{v\in\Sigma_m^{(k+1)}} \binom{m}{v_1, \ldots, v_{k+1}} \left(\prod_{i=1}^{k+1} x_i^{v_i}\right) v_j.\]
Therefore, Equation~\eqref{eq:Sx} implies
\[
Sx = \sum_{v\in\Sigma_m^{(k+1)}} \binom{m}{v_1, \ldots, v_{k+1}} 
\left(\prod_{i=1}^{k+1} x_i^{v_i}\,\indi_{v_i\neq 0}\right) \left(R(v)-\sigma v\right) + \sigma m x,
\]
and thus, for all $x = x_1 \bs e_1 +  \cdots + x_{k+1}\bs e_{k+1}\in\Sigma^{(d)}$,
\[\sum_{v\in\Sigma_m^{(k+1)}} \binom{m}{v_1, \ldots, v_{k+1}} 
\left(\prod_{i=1}^{k+1} x_i^{v_i}\,\indi_{v_i\neq 0}\right) \left(R(v)-\sigma v\right)=0.\]
Take $x_p = t\in(0,1)$ and $x_i= (1-t)/k$ for all $i\in\{1, \ldots, k+1\}\setminus\{p\}$.
To all $v\in\Sigma_m^{\sss (k+1)}$, we associate bijectively the couple $(v_p, v^*_p)$, 
where $v_p$ is the $p$-th coordinate of $v$ and where 
$v^*_p\in\Sigma_{m-v_p}^{\sss (k)}$ is the vector whose~$k$ coordinates are~$v_1, \ldots, v_{p-1}, v_{p+1}, \ldots, v_{k+1}$. 
With these notations, we have 
\begin{align*}
0
&=\sum_{v\in\Sigma_m^{(k+1)}} \binom{m}{v_1, \ldots, v_{k+1}} (\nicefrac t k)^{v_p} (1-t)^{m-v_p} \left(R_p(v)-\sigma v_p\right)\\
&=\sum_{v_p=1}^m \sum_{v^*_p\in\Sigma^{(k)}_{m-v_p}} \binom m {v_1, \ldots, v_{k+1}}  (\nicefrac t k)^{v_p} \left(R_p(v)-\sigma v_p\right) \sum_{\ell=0}^{m-v_p} \binom{m-v_p}{\ell}(-t)^{\ell}\\
&=\sum_{u=1}^{m} t^u \sum_{v_p=1}^{u} \sum_{v^*_p\in\Sigma^{(k)}_{m-v_p}} \binom m {v_1, \ldots, v_{k+1}} \binom{m-v_p}{u-v_p}  k^{-v_p} \left(R_p(v)-\sigma v_p\right)  (-1)^{u-v_p},
\end{align*}
implying that, for all $1\leq u\leq m$,
\[\sum_{v_p=1}^u \binom{m-v_p}{u-v_p} k^{-v_p} (-1)^{u-v_p}
\sum_{v^*_p\in\Sigma^{(k)}_{m-v_p}} \binom m {v_1, \ldots, v_{k+1}} \left(R_p(v)-\sigma v_p\right) = 0.\] 
This equation for $u=1$ gives that $R_p(v) = \sigma$ for all vector $v$ such that $v_p=1$. 
Using the above equation for $u=2$, one can then induce that $R_p(v) = 2\sigma$ for all vector $v$ such that $v_p = 2$, and, inductively, 
prove that $R_p(v)=\sigma v_p$ for all $1\leq p\leq k+1$.

In total, for all $v\in\Sigma_m^{(k+1)}$, we have $R(v) = \sigma v$, which concludes the induction argument.
\end{proof}

\begin{proof}[Proof of Theorem~\ref{th:diag}]
We have that (since we assume that the urn is balanced)
\begin{equation}\label{eq:diag}
Z_{n+1} = Z_n + \frac1{T_0 + (n+1)S} \big(\Delta M_{n+1} + \varepsilon_{n+1}\big),
\end{equation}
where we recall that $\Delta M_{n+1}$ is the increment of a martingale 
and $\|\varepsilon_{n+1}\| = \mathcal O(\nicefrac1n)$ 
almost surely when $n$ tends to infinity.
We infer that, for all $1\leq k\leq d$,
\begin{equation}\label{eq:martingale}
Z_{n,k} = Z_{0,k} + \sum_{i=0}^{n-1} \frac{\Delta M_{i+1,k}}{T_0+ (i+1)S} + \sum_{i=0}^{n-1} \frac{\varepsilon_{i+1,k}}{T_0+(i+1)S}.
\end{equation}
Note that the last sum of this last equation converges almost surely since either $\varepsilon_i = 0$ for all $i$, 
or $|\varepsilon_{i,k}|\leq \|\varepsilon_i\|=\mathcal O(\nicefrac1i)$ almost surely.
The first sum in Equation~\eqref{eq:martingale} is a martingale and its quadratic variation is given by
\[\mathtt{Hook}_n := \sum_{i=0}^{n-1} \frac{\mathbb E[\Delta M_{i+1,k}^2 | \mathcal F_i]}{(T_0+ (i+1)S)^2}.\]
Recall that $\Delta M_{i+1,k}^2$ is almost surely bounded by $4(\|R\|+S)^2$ implying that $\mathtt{Hook}_n$ is almost surely convergent when $n$ goes to infinity. 
Therefore, the martingale itself, i.e. the first sum in Equation~\eqref{eq:martingale} converges almost surely to a finite random variable.
In total, $Z_n$ converges almost surely to a random vector $Z_{\infty}$.
\end{proof}

\section{The two-colour particular case (proof of Corollary~\ref{th:2col})}\label{sec:two_colour}
In the two-colour balanced case, the renormalised composition vector $Z_n$ of the urn at time $n$ is actually characterised by its first coordinate since $Z_{n,1}+Z_{n,2}=1$. Thus, it is enough to study this first coordinate, and the problem becomes unidimensional and thus slightly easier. Although Corollary~\ref{th:2col} could be deduced from Theorem~\ref{th:main}, we give here a stand-alone and much simpler proof:
we apply directly a result about one-dimensional stochastic algorithms (see~\cite[Corollary 2.7 and Theorem~2.9]{Pemantle}).
Let
\[g(x):= \sum_{k=0}^m \binom{m}{k} x^k (1-x)^{m-k} (a_{m-k}-S x),\]
then
\[X_{n+1}-X_n = \frac1{n}\left(\frac{g(X_n)}{S} + \Delta M_{n+1} + \varepsilon_{n+1}\right),\]
where
\[\Delta M_{n+1} = \frac1S\big(Y_{n+1}- \mathbb E [Y_{n+1}|\mathcal F_n]\big),\]
with
$Y_{n+1} := a_{m-\zeta_{n+1}} - S Z_n$,
and where
$\varepsilon_{n+1}=0$ in the {\it with}-replacement case
and $\varepsilon_{n+1}$ satisfies $|\varepsilon_n| = \mathcal O(\nicefrac1n)$ almost surely when $n$ tends to infinity
in the {\it without}-replacement case.

Since $g(x)$ is a polynomial, either it is identically null, or it has isolated zeros.
The case $g\equiv 0$ is called the {\it diagonal} case and is treated separately.
If $g\not\equiv 0$, then $g$ has at most $m$ zeros on $[0,1]$.
In view of Corollary~2.7 and Theorem~2.9 in~\cite{Pemantle}, we can conclude that
$X_n$ converges almost surely to one of the zeros of $g$ where $g'$ is non-positive,
depending on the initial composition of the urn.
This proves Corollary~\ref{th:2col}$(a)$.

If we assume that $X_n$ converges almost surely to some stable zero $\theta_\star$ of $g$, 
then we can apply Theorem~\ref{CLT-Pages}. We have
\[\mathbb E[(\Delta M_{n+1})^2 | \mathcal F_n] \to 
\Gamma := \frac1{S^2} \sum_{k=0}^m \binom mk \theta_\star^k (1-\theta_\star)^{m-k} (a_{m-k}-S\theta_\star)^2.\]
Finally let $\Lambda = -\nicefrac{g'(\theta_\star)}{S}$, and the following central limit theorem holds:
\begin{itemize}
\item If $\Lambda >\nicefrac 12$, then $\sqrt n(X_n - \theta_\star) \stackrel{n\to \infty}{\longrightarrow} \mathcal N\big(0,\frac{\Gamma}{2\Lambda-1}\big)$, in distribution.
\item If $\Lambda = \nicefrac 12$, then $\sqrt{\nicefrac{n}{\log n}} (X_n-\theta_\star) \stackrel{n\to \infty}{\longrightarrow} \mathcal N(0,\Gamma)$, in distribution.
\item If $\Lambda <\nicefrac 12$, then $n^{\Lambda}(X_n - \theta_\star)$ converges almost almost surely to a random variable.
\end{itemize}
This concludes the proof of Corollary~\ref{th:2col}.

{\bf Diagonal case. }
In the two-colour case, assuming both that the urn is balanced and that $g\equiv 0$ implies that
there exists an integer $q$ such that
\[R=
\begin{pmatrix}
mq & 0\\
(m-1)q & q\\
\vdots & \vdots\\
0 & mq\\
\end{pmatrix}.
\]
This case is treated by Kuba and Mahmoud~\cite{KM} 
and Kuba and Sulzbach~\cite{KS}.

\section{Examples}\label{sec:examples}

\subsection{Two-colour examples}
The first four examples are two-colour examples: we thus apply Corollary~\ref{th:2col} to them.
We choose to focus on non-affine examples to which the results of Kuba \& Mahmoud~\cite{KM} do not apply;
note that affine models correspond to the case when $g$ is linear.

\vspace{\baselineskip}
{\bf Example 4.1.1:}
Take
\[R =
\begin{pmatrix}
1 & 2\\
2 & 1\\
1 & 2
\end{pmatrix}.\]
In that case, $g(x)=(1+x)(1-2x)$, whose single zero in $[0,1]$ is $\nicefrac12$.
Thus, almost surely when $n$ goes to infinity, 
the proportion of white balls in the urn at time $n$ converges 
to $\nicefrac12$.
In this particular case, $S=3$, and $\Lambda=-g'(\nicefrac12)/S = 1 > \nicefrac12$.
One can check that $\Gamma=\nicefrac1{36}$, and since $2\Lambda-1 = 1$, we get, by Corollary~\ref{th:2col}$(b)$,
\[\sqrt n(X_n - \nicefrac12) \to \mathcal N(1,\nicefrac1{36}),\]
in distribution when $n$ tends to infinity.

\vspace{\baselineskip}
{\bf Example 4.1.2:}
Take
\[R =
\begin{pmatrix}
4 & 0\\
1 & 3\\
1 & 3
\end{pmatrix}.\]
In that case, $g(x)=(1-x)(1-3x)$, whose two zeros in $[0,1]$ are $1$ and $\nicefrac13$.
Note that $g'(1)=2$ and $g'(\nicefrac13)=-2$, thus, almost surely when $n$ tends to infinity,
the proportion of white balls in the urn converges to $\nicefrac13$.
Note that $\Lambda=-\nicefrac{g'(\nicefrac13)}{S} = \nicefrac12$.
Also, one can check that $\Gamma=\nicefrac1{18}$, 
which implies, by Corollary~\ref{th:2col}$(b)$,
\[\sqrt{\nicefrac n{\log n}}(X_n - \nicefrac13) \to \mathcal N(1,\nicefrac1{18}),\]
in distribution when $n$ tends to infinity.

{\bf Example 4.1.3:} Take
\[R =
\begin{pmatrix}
7 & 1\\
3 & 5\\
1 & 7
\end{pmatrix}.\]
Then $g(x) = 2x^2-4x+1$ and this polynomial
has a unique root in $[0,1]$, given by $\theta_\star=1-\nicefrac{\sqrt 2}2$.
Applying Theorem~\ref{th:main}$(a)$, we get that the proportion of white balls $X_n$ converges almost
surely to $1-\nicefrac{\sqrt 2}2$. 
Moreover, $\Lambda = -g'(\nicefrac12)/8 = \nicefrac{\sqrt2}4 < \nicefrac 12$.
Thus, by Corollary~\ref{th:2col}$(b)$, there exists a random variable $\Psi$ such that
\[n^{\nicefrac{\sqrt 2}4}(X_n-\nicefrac12)\to \Psi,\]
almost surely when $n$ tends to infinity.

\vspace{\baselineskip}
{\bf Example 4.1.4:}
Take
\[R =
\begin{pmatrix}
6 & 0\\
3 & 3\\
1 & 5
\end{pmatrix}.\]
In this case, $g(x)=(x-1)^2$, implying that $1$ is the unique zero of $g$ in $[0,1]$.
Thus, the proportion of white balls in the urn converges almost surely to $1$.
We have $g'(1)=0$, which makes us unable to apply Corollary~\ref{th:2col}$(b)$: 
we have no information about the speed of convergence of the proportion of white balls to its limit.

\vspace{\baselineskip}
\begin{figure}
\begin{center}
\includegraphics[width=7cm]{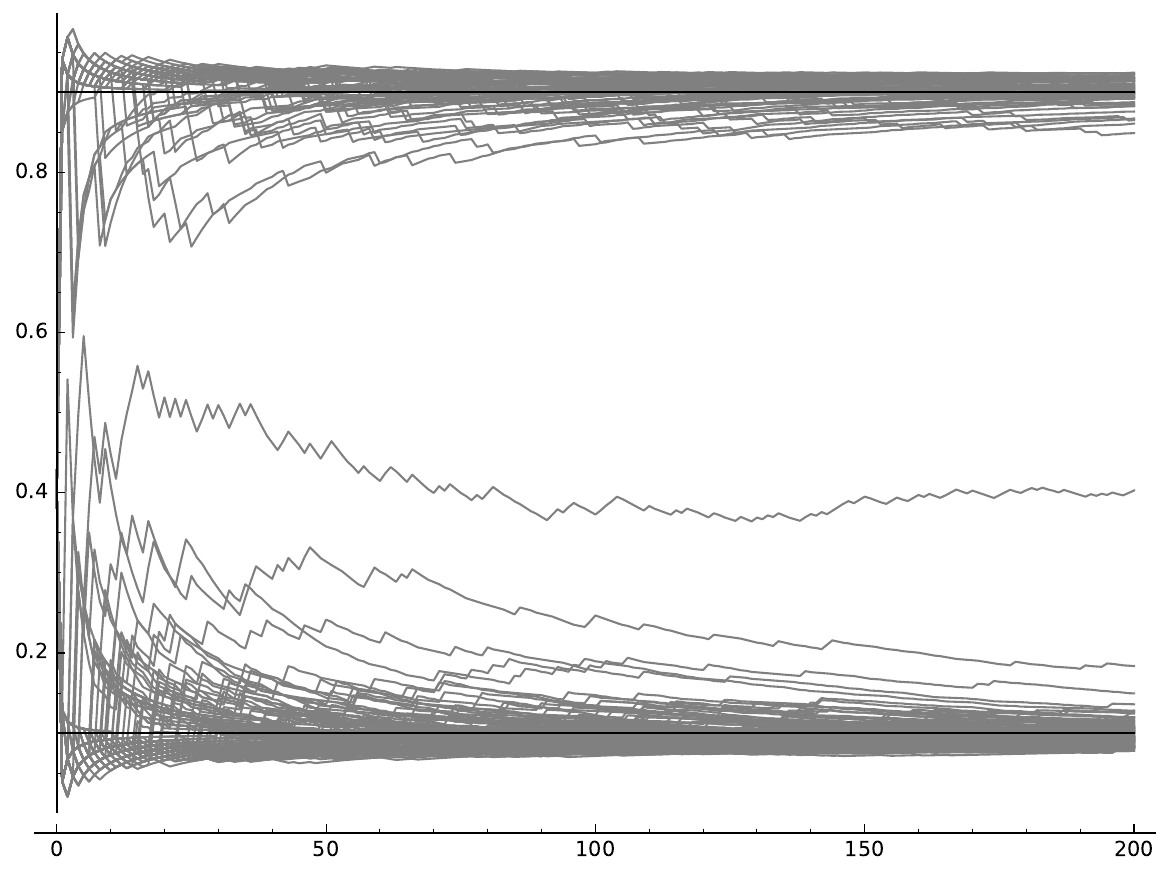}
\caption{A hundred realisations of the urn process of Example 4.1.5, all starting at $(4,6)$ and run for 200 steps. 
One can see that each trajectory converges to one of the two stable zeros of $g$, namely $\nicefrac1{10}$ and $\nicefrac9{10}$.}
\label{fig:46}
\end{center}
\end{figure}
{\bf Example 4.1.5:} 
Let us consider an example for which $m=3$:
\[R = 
\begin{pmatrix}
82 & 9\\
91 & 0\\
0 & 91\\
9 & 82
\end{pmatrix}.
\]
In that case, $S=91$, and $g(x) = -200(x-\nicefrac1{10})(x-\nicefrac12)(x-\nicefrac{9}{10})$ admits three zeros on $[0,1]$, 
namely $\nicefrac1{10}, \nicefrac12$, and $\nicefrac9{10}$.
Note that $g'(\nicefrac12)>0$ while $\nicefrac1{10}$ and $\nicefrac{9}{10}$ are stable zeros of $g$.
Therefore, the proportion of white balls in the urn converges to some random variable $\theta_\star$, 
and $\theta_\star\in\{\nicefrac1{10}, \nicefrac9{10}\}$ almost surely (see Figure~\ref{fig:46}).
Note that $g'(\nicefrac13)=g'(\nicefrac23)=-64$ implying that $\Lambda=\nicefrac{64}{91}>\nicefrac12$.
Therefore, one can check that
$\sqrt n(X_n - \theta_\star)$ converges in distribution to a centred normal distribution of variance $\nicefrac{4131}{67340}$.

Note that although we know that $X_n$ goes almost surely to a random variable $X_{\infty}$ whose support is $\{\nicefrac1{10}, \nicefrac9{10}\}$, it is an open problem to determine the distribution of $X_{\infty}$, and it is expected that this distribution depends on the initial composition of the urn $X_0$.

\subsection{Three-colour examples}
For all three-colour examples, we use a computer algebra system to help with the calculations.
Note that it is in practise harder to apply Theorem~\ref{th:main}$(a)$ than Theorem~\ref{th:main}$(b)$;
the strategy is to first find the zeros of the function $h$, then calculate the Jacobian of $h$ at these zeros, 
extract their spectrum and isolate the zeros whose Jacobians only have negative eigenvalues. 
In all the examples below, there is at most one such zero and
it is our candidate for the limit of the normalised composition vector $Z_n$;
if there is no such zero, we cannot say more about the considered urn (see example 4.2.5) 
except that its limit set is a compact connected set stable by the flow of~$\dot{x} = h(x)$.

Once we have a candidate $\theta$ to be the limit of $Z_n$, 
we need to check (A1) and (A2). 
All the examples below are given for $m=2$, 
since it implies that $\langle h(z), z-x\rangle$ for $x\in\mathcal Z(h)$
is a polynomial of order at most~$3$ and the set where it is negative (for (A1)) or positive (for (A2-i)) can be calculated exactly (we do it with a computer algebra system, but this can be done by hand on simple examples). Checking (A2-ii) is usually straightforward, and checking (A2-iii) is done using Borel-Cantelli's lemma.

{\bf Example~4.2.1:} We consider the three-colour urn scheme defined by the following replacement rule:
\[\begin{matrix}
R(2,0,0)=(1,0,0); & & R(0,1,1)=(1,0,0);\\
R(0,2,0)=(0,1,0); & & R(1,0,1)=(0,1,0);\\
R(0,0,2)=(0,0,1); & & R(1,1,0)=(0,0,1).\\
\end{matrix}\]
In that case, 
\[h(x) = \begin{pmatrix}
x_1^2 + 2x_2 x_3 - x_1\\
x_2^2 + 2x_1 x_3 - x_2\\
x_3^2 + 2x_2 x_3 - x_3
\end{pmatrix}.\]
The function $h(x)$
admits four zeros on the simplex $\Sigma_2^{\scriptscriptstyle (3)}$ 
being given by $(1,0,0)$, $(0,1,0)$, $(0,0,1)$ and $(\nicefrac13,\nicefrac13,\nicefrac13)$.

We recall that the function $h$ is a function from $\Sigma^{\sss (d)}$ onto $\{(x_1, x_2, x_3)\in\mathbb R^d \colon x_1+x_2+x_3 = 0\}$, two spaces of dimension~2, therefore, the Jacobian matrix of $h$ at any point of $\Sigma^{\sss (d)}$ is a $2\times 2$ matrix.
The eigenvalues of the Jacobian matrix of $h$ at the first three zeros mentioned above are $1$ and $-3$. 
The eigenvalues of $\nabla h(\nicefrac13,\nicefrac13,\nicefrac13)$ are $-1$ (with multiplicity $2$), thus this zero is stable and is our candidate to be the almost sure limit of $Z_n$. We actually have $\nabla h(\nicefrac13,\nicefrac13,\nicefrac13)=-\mathtt{Id}_2$. 

One can check that
the only solutions of $\langle h(x), x-\theta \rangle = 0$ on $\Sigma^{\scriptscriptstyle (3)}$, 
where $\theta = (\nicefrac13,\nicefrac13,\nicefrac13)$, are the four zeros of~$h$.
Since $\theta$ is a stable zero of $h$, $\langle h(x), x-\theta \rangle<0$  in a neighbourhood of $\theta$, 
and thus, by continuity, $\langle h(x), x-\theta \rangle < 0$ for all $x$ in $\Sigma^{\scriptscriptstyle (3)}$ 
such that $h(x)\neq 0$.

One can check that
the only solutions of $\langle h(x), x-\theta \rangle = 0$ on $\Sigma^{\scriptscriptstyle (3)}$, 
where $\theta = (\nicefrac13,\nicefrac13,\nicefrac13)$, are the four zeros of~$h$.
This implies (A1).

To prove (A2-i), first note that, by symmetry, it is enough to prove that $\langle h(u), u-(1,0,0)\rangle\geq 0$ in a neighbourhood of $(1,0,1)$ in $\Sigma^{\sss (d)}$.
A straightforward calculation gives
\[\langle h(u), u-(1,0,0)\rangle = 3u_2u_3(2-3(u_2+u_3)),\]
and this quantity is non-negative as soon as $u_2+u_3\leq \nicefrac23$,
which implies Assumption (A2-i).

For (A2-ii), note that, as long as $Z_0\notin \mathcal Z(h)\setminus\{\theta\}$, i.e. as long as the urn contains balls of more then one colour at time zero, we have that
$Z_n\notin \mathcal Z(h)\setminus\{\theta\}$ for all $n\geq 0$.

Finally, in the without-replacement case, we use Borel-Cantelli lemma to show that $U_{n,i}\to+\infty$ almost surely as $n\to+\infty$. 
Indeed, let $\mathcal E_{i,n}$ the event that at least one ball of colour $i$ as added to the urn at time~$n+1$. For $\mathcal E_{i,n}$ to occur, the handful drawn at time $n$ needs to be $(2, 0, 0)$ or $(0,1,1)$, which implies
\[\mathbb P(\mathcal E_{1,n}\mid \mathcal F_n) 
= \frac{U_{n,1}(U_{n,1}-1)+2U_{n,2}U_{n,3}}{T_n(T_n-1)}.\]
Note that $T_n= T_0+2n\geq 2n$, which implies $\min(U_{n,1}, U_{n,2}, U_{n,3})\geq \nicefrac{2n}3$.
Furthermore, under the assumption that at time zero, there is at least one ball of each colour, we have $U_{n,i}\geq 1$ for all $n\geq 0$, for all $1\leq i\leq 3$, and thus
$\mathbb P(\mathcal E_{1,n}) \geq (\nicefrac13+o(1))/n$ as $n\to+\infty$.
Therefore, by Borel-Cantelli lemma, $U_{n,1}\to+\infty$ almost surely when $n\to+\infty$.
By symmetry, this is also the case for $U_{n,2}$ and $U_{n,3}$, and (A2-iii) holds.

Thus, if at time zero there is at least one ball of each colour in the urn, then all assumptions of Theorem~1$(a)$ are satisfied and $Z_n \to (\nicefrac13, \nicefrac13, \nicefrac13)$ almost surely when $n\to+\infty$ in both the with- and the without-replacement cases. In the with-replacement case, one can replace the assumption that $\min(U_{n,1}, U_{n,2}, U_{n,3})\geq 1$ by $\max(U_{0,2}, U_{0,3})\geq 1$ (because (A2-iii) is only required in the without-replacement case).

We now apply Theorem~\ref{th:main}$(b)$.
One can check that
\begin{equation}\label{eq:Gamma_ex1}
\Gamma=\frac19 \begin{pmatrix}
2 & -1 & -1\\
-1 & 2 & -1\\
-1 & -1 & 2\\
\end{pmatrix}
\end{equation}
and
\[\Sigma= \int_0^{+\infty} (\mathtt e^{-\nicefrac u2} \mathtt{Id})^t\,\Gamma (\mathtt e^{-\nicefrac u2}\mathtt{Id})\,du = \Gamma.\]
Finally note that the eigenvalue of $-\nabla h(\theta)$ with the smallest real part is $\Lambda=1>\nicefrac S2=\nicefrac12$. We have that $2\mathtt{Re}(\Lambda)/S-1 = 1$, implying that,
in distribution when $n$ tends to infinity,
\[\sqrt n(Z_n - \nicefrac{\boldsymbol 1}{3}) \to \mathcal N(0,\Gamma),\]
where $\boldsymbol 1 = (1,1,1)$ and $\Gamma$ given by Equation~\eqref{eq:Gamma_ex1}.

Since $Z_n$ and $\nicefrac{\boldsymbol 1}{3}$ are in the simplex $\Sigma^{\sss (d)}$,
we know that $\Sigma = \Gamma$ are of rank at most~2, and that, in particular,
\[\Sigma\cdot (1,1,1)^t = \Gamma\cdot (1,1,1)^t = 0,\]
which can indeed be checked using~\eqref{eq:Gamma_ex1}.

\vspace{\baselineskip}
{\bf Example~4.2.2:} Consider the three-colour urn scheme defined by the following replacement rule:
\[\begin{matrix}
R(2,0,0)=(2,0,0); & & R(0,1,1)=(0,1,1);\\
R(0,2,0)=(1,0,1); & & R(1,0,1)=(0,2,0);\\
R(0,0,2)=(1,1,0); & & R(1,1,0)=(0,0,2).\\
\end{matrix}\]
In that case, 
\[h(x) = \begin{pmatrix}
2x_1^2 + x_2^2 + x_3^3 - 2x_1\\
x_3(4x_1+2x_2+x_3) -2x_2\\
x_2(4x_1+x_2+2x_3) - 2x_3
\end{pmatrix}.\]
One can check that $h(x)$
admits two zeros on the simplex $\Sigma_2^{(3)}$ 
being given by $(1,0,0)$, $(\nicefrac15,\nicefrac25,\nicefrac25)$.
The eigenvalues of the Jacobian matrix of $h$ at $(1,0,0)$ are $2$ and $-6$.
The eigenvalues of $\nabla h(\nicefrac15,\nicefrac25,\nicefrac25)$
are $-2$ and $-\nicefrac{18}5$, 
thus this zero is stable and is our candidate to be the almost sure limit of~$Z_n$.

Using a computer algebra software, we can also check that the only solutions 
of $\langle h(x), x-\theta\rangle = 0$ are $(1,0,0)$ and $\theta = (\nicefrac15, \nicefrac25, \nicefrac25)$.
Since $\theta$ is a stable zero, $\langle h(x), x-\theta \rangle$ is negative in a neighbourhood of $\theta$, 
and, by continuity, it is negative on $\Sigma^{\scriptscriptstyle (3)}\setminus\{(1,0,0), \theta\}$.

Unfortunately, (A2-i) does not hold for this example, and we thus cannot apply Theorem~1$(a)$.
However, simulations (see Figure~\ref{fig:1}) seem to indicate that, indeed, $Z_n\to\theta$.

\begin{figure}
\vspace{10pt}
\centering
\includegraphics[trim= 0 0 0 1.5cm, width=7cm]{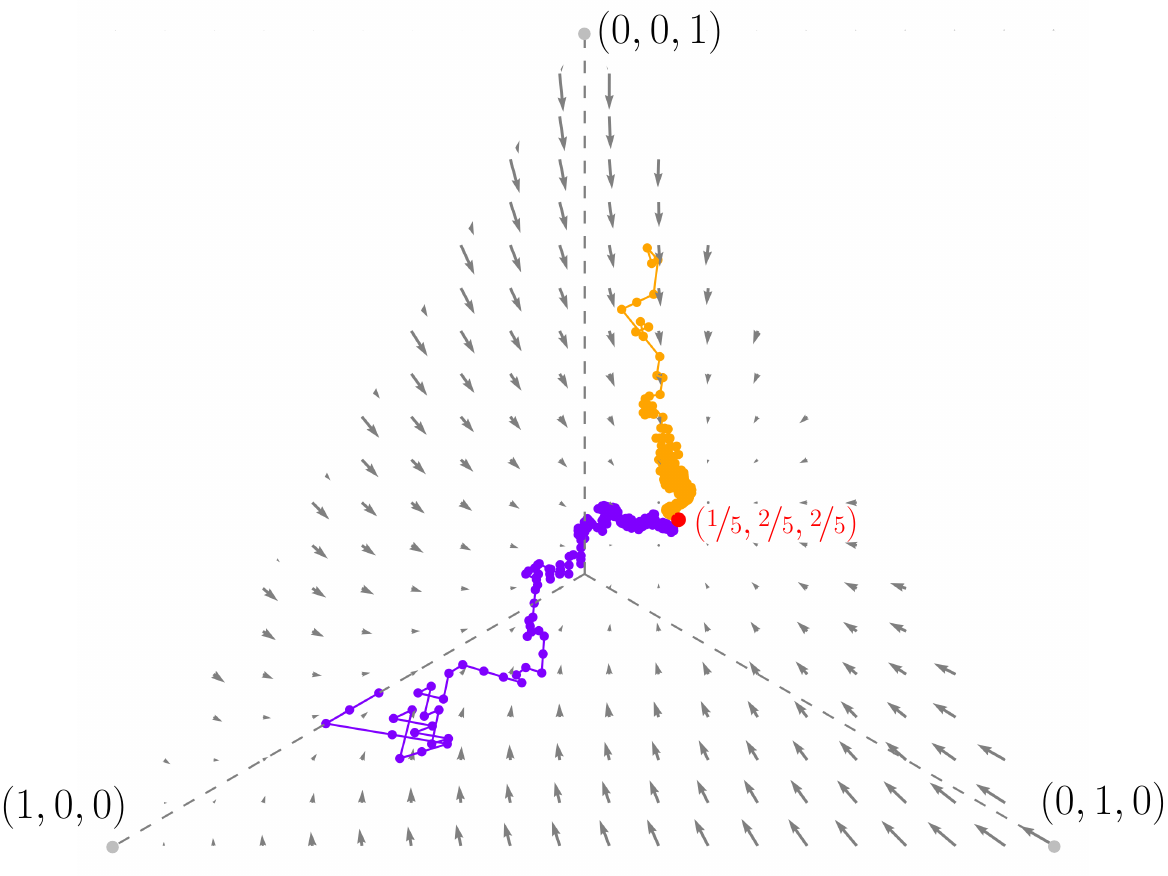}
\caption{Two realisations of the urn of Example~4.2.2 starting respectively from $(10, 3, 3)$ and $(2, 6, 20)$, and run for 400 steps. 
The gray arrows are the vector field associated to~$h$ on the simplex $\Sigma^{\sss (3)}$; one can in particular see that $(1,0,0)$ is indeed an unstable zero while $(\nicefrac15, \nicefrac25, \nicefrac25)$ is stable.}
\label{fig:1}
\end{figure}

We now apply Theorem~\ref{th:main}$(b)$ on the event that $Z_n\to \theta$.
One can calculate that
\begin{equation}
\label{eq:Sigma_ex2}
\Gamma=\frac1{25} \begin{pmatrix}
2 & -1 & -1\\
-1 & 3 & -2\\
-1 & -2 & 3\\
\end{pmatrix},\quad
\text{ and }\quad
\Sigma= \frac1{25}
\begin{pmatrix}
2 & -1 & -1\\
-1 & 19 & -6\\
-1 & -6 & 19\\
\end{pmatrix}.
\end{equation}
Finally note that the eigenvalue of $-\nabla h(\theta)$ with the smallest real part is $\Lambda:=2>\nicefrac S2=1$, implying that, on the event $Z_n\to\theta$,
in distribution when $n$ tends to infinity, since $\nicefrac{2\mathtt{Re}(\Lambda)}S - 1 = 1$,
\[\sqrt n(Z_n - (\nicefrac15,\nicefrac25,\nicefrac25)^t) \to \mathcal N(0, \Sigma),\]
where $\Sigma$ is given by Equation~\eqref{eq:Sigma_ex2}. 
(Again, one can check that, as expected, $\Sigma\cdot (1,1,1)^t = 0$.)

\vspace{\baselineskip}
{\bf Example~4.2.3:}
Consider the three-colour urn scheme defined by the following replacement rule:
\[\begin{matrix}
R(2,0,0)=(3,0,0); & & R(0,1,1)=(3,0,0);\\
R(0,2,0)=(0,3,0); & & R(1,0,1)=(1,1,1);\\
R(0,0,2)=(0,0,3); & & R(1,1,0)=(1,1,1).\\
\end{matrix}\]
\begin{figure}[!b]
\begin{center}
\includegraphics[width=8cm]{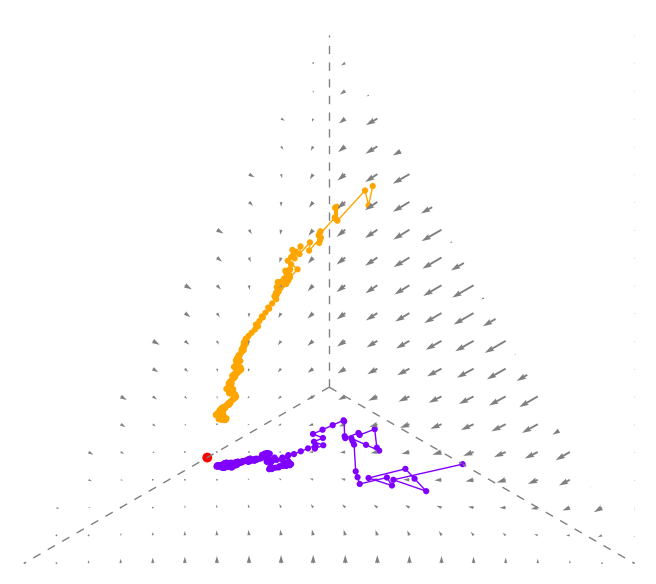}
\includegraphics[width=8cm]{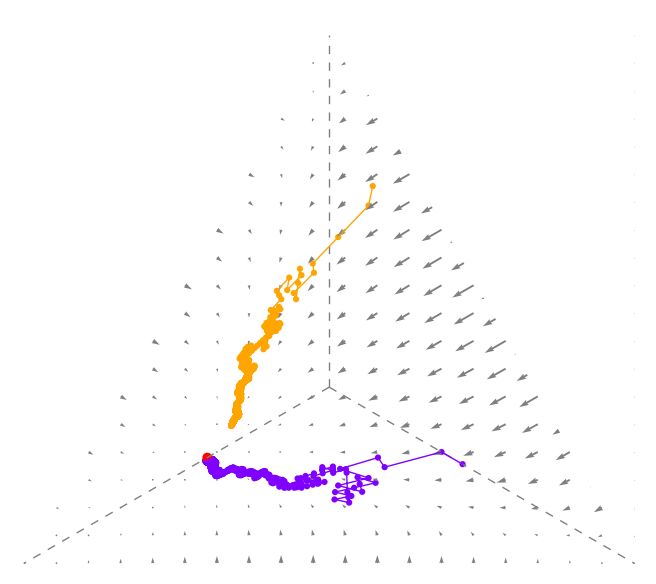}
\end{center}
\caption{Two realisations of the urn of Example 4.2.3 starting respectively from $(3, 10, 3)$ and $(2, 5, 20)$, 
run for 400 steps (left) and 4000 steps (right). Both realisations converge to the stable zero $(\nicefrac15,\nicefrac25,\nicefrac25)$;
one can compare the speed of convergence with the faster Example 4.2.2 displayed in Figure~\ref{fig:1}.}
\end{figure}
In that case, 
\[h(x)=\begin{pmatrix}
3 x_1^2 + 2x_1 x_2 + 2 x_1 x_3 + 6 x_2 x_3 - 3x_1\\
3 x_2^2 + 2 x_1 x_2 + 2 x_1 x_3 - 3 x_2\\
3 x_3^2 + 2 x_1 x_2 + 2 x_1 x_3 - 3 x_3\\
\end{pmatrix}
\]
one can check that $h(x)$ admits four zeros on the simplex $\Sigma_2^{(3)}$, 
given by $(1,0,0)$, $(0,1,0)$, $(0,0,1)$ and $(\nicefrac35, \nicefrac15, \nicefrac15)$. 
The eigenvalues of $\nabla h$ at the first three zeros are $-2-\sqrt{13}$ and $-2+\sqrt{13}$.
The eigenvalues of $\nabla h(\nicefrac35, \nicefrac15, \nicefrac15)$ are $-1$ and $-\nicefrac95$. 
Thus $\theta = (\nicefrac35, \nicefrac15, \nicefrac15)$ is a stable zero of $h$.

One can also check that the only solutions of $\langle h(x), x- \theta \rangle = 0$ 
on $\Sigma^{\scriptscriptstyle (3)}$ are the zeros of $h$. 
Using the fact that $\langle h(x), x- \theta \rangle<0$ on a neighbourhood of $\theta$, thus on the whole domain except at the zeros of $h$, we get that (A1) is satisfied.

One can check that for all $u\in\Sigma^{\sss (d)}$,
\[\langle h(u), u-(1,0,0)\rangle = 3u_2u_3(2-5(1-u_1))\geq 0\]
in a neighbourhood of $(1,0,0)$.
We also have
\[\langle h(u), u-(0,1,0)\rangle =  3u_1u_3(4-5(1-u_2))\geq 0\]
in a neighbourhood of $(0,1,0)$; 
by symmetry, the analogue statement holds for $(0,0,1)$,
which concludes the proof of (A2-i).

If at time zero, there is at least one ball of colour 2 or 3 in the urn, 
then $Z_n\notin\mathcal Z(h)\setminus\{\theta\}$, and thus (A2-ii) holds.

For (A2-iii), we use Borel-Cantelli's lemma: let $\mathcal E_{n,i}$ be the probability that at least one ball of colour $i$ is added in the urn at time $n+1$. In the without-replacement case,
\[\mathbb P(\mathcal E_{1,n}\mid \mathcal F_n)
= \frac{U_{n,1}(U_{n,1}-1)+U_{n,1}U_{n,2} + U_{n,1}U_{n,3}+U_{n,2}U_{n,3}}{T_n(T_n-1)}.\]
Because the urn is balanced, 
we have $T_n = T_0+ 3n\geq 3n$, implying that $\min(U_{n,1}, U_{n,2}, U_{n,3})\geq n$.
If we assume that, at time zero, there is at least one ball of colour 2 and one ball of colour 3, then
$\mathbb P(\mathcal E_{1,n}\mid \mathcal F_n)\geq (\nicefrac19+o(1))/n$, which, by Borel-Cantelli's lemma, implies that $U_{n,1}\to+\infty$ almost surely as $n\to+\infty$.
Similar arguments implie that $U_{n,2}$ and $U_{n,3}$ also diverge almost surely as long as $\min(U_{0,2}, U_{0,3})\geq 1$.

Thus Theorem 1$(a)$ implies and that, 
if $\min(U_{0,2}, U_{0,3})\geq 1$, then almost surely when $n\to+\infty$,
$Z_n \to (\nicefrac35, \nicefrac15, \nicefrac15)$ in both the with- and the without-replacement cases.
In the with-replacement case, one can replace the assumption $\min(U_{0,2}, U_{0,3})\geq 1$ by $\max(U_{0,2}, U_{0,3})\geq 1$ (because (A2-iii) is trivial in this case).

We can now apply Theorem~\ref{th:main}$(b)$. 
In this case, $\Lambda = 1<\nicefrac S2 = \nicefrac32$, 
and the Jordan block of $\nabla h(\theta)$ associated to~1 has size~1 
(since $\nabla h$ has two distinct eigenvalues given by $-1$ and $-\nicefrac95$),
implying that $n^{\nicefrac13}(Z_n-(\nicefrac35, \nicefrac15, \nicefrac15)^t)$ converges almost surely to a random variable $\Psi$.

\vspace{\baselineskip}
{\bf Example 4.2.4:}
Consider the three colour urn scheme defined by the following replacement rule:
\[\begin{matrix}
R(2,0,0)=(0,0,2); & & R(0,1,1)=(0,1,1);\\
R(0,2,0)=(0,0,2); & & R(1,0,1)=(1,0,1);\\
R(0,0,2)=(0,0,2); & & R(1,1,0)=(1,1,0).\\
\end{matrix}\]
We have
\[h(x)=2\begin{pmatrix}
x_1x_2 + x_1x_3 - x_1\\
x_1x_2 + x_2 x_3 - x_2\\
x_1^2 + x_2^2 + x_3^2 + x_1x_3 + x_2 x_3 - x_3
\end{pmatrix}.\]
The function $h$ admits a unique zero on the simplex $\Sigma_2^{(3)}$ given by $(0,0,1)$.
The eigenvalues of $\nabla h(0,0,1)$ are $0$ (with multiplicity 2).

We can solve explicitly $\langle h(x), x-(0,0,1)\rangle$ on $\Sigma^{\scriptscriptstyle (3)}$ and check that $(0,0,1)$ is the only solution. Also, $\langle h(1,0,0), (1,0,0)-(0,0,1)\rangle = -4$. 
By continuity, $\langle  h(x), x-(0,0,1)\rangle <0$ for all $x\in\Sigma^{\scriptscriptstyle (3)}\setminus\{(0,0,1)\}$, meaning that (A1) is satisfied. (A2) holds trivially since $\mathcal Z(h)\setminus\{\theta\} = \varnothing$. Thus Theorem~\ref{th:main}$(a)$ applies, and 
the renormalised composition vector $Z_n$ converges almost surely to $(0,0,1)$.

Nothing can be said about the rate of convergence since $\nabla h(0,0,1)$ admits zero as an eigenvalue (and Theorem~\ref{th:main}$(b)$ thus does not apply).

\vspace{\baselineskip}
{\bf Example 4.2.5: Rock-paper-scissor game~\cite{LL}.}
Consider the three colour urn scheme defined by the following replacement rule:
\[\begin{matrix}
R(2,0,0)=(1,0,0); & & R(0,1,1)=(0,1,0);\\
R(0,2,0)=(0,1,0); & & R(1,0,1)=(0,0,1);\\
R(0,0,2)=(0,0,1); & & R(1,1,0)=(1,0,0).\\
\end{matrix}\]
This urn scheme is studied in Laslier \& Laslier~\cite{LL} as a model for the famous {\it rock-paper-scissors} game;
the first colour represents {\it scissors}, the second {\it paper} and the third {\it rock},
and the replacement rule encodes the standard rules of the game 
(scissors wins against paper, paper wins against rock and rock wins against scissors).
We have
\[h(x) = \begin{pmatrix}
x_1(x_1 + 2x_2-1)\\
x_2(x_2 + 2x_3-1)\\
x_3(x_3 + 2x_1 -1)
\end{pmatrix}.\]
The function $h$ admits four zeros, being $(1,0,0)$, $(0,1,0)$, $(0,0,1)$ and $(\nicefrac13, \nicefrac13, \nicefrac13)$, but none of them is stable.

In~\cite[Theorem~7]{LL} it is proved that the composition vector does not converge when $n$ goes to infinity, but concentrate on a cycle instead. Therefore, our Theorem~\ref{th:main} does not apply further than saying that the limit set of $(Z_n)_{n\geq 0}$ is a compact connected set.

The authors also study what the call a {\it three-alternative} scheme, in which a sample of three balls is sampled at every time step, and the replacement rule depends on the order of sampling of those balls: if one samples {\it (scissors, rock, paper)}, we get 
\begin{center}(\textit{scissors} against (\textit{rock} against \textit{paper})) = \textit{scissors} against \textit{paper} = \textit{scissors},\end{center}
while the sample {\it (paper, scissors, rock)} gives
\begin{center}(\textit{paper} against (\textit{scissors} against \textit{rock})) = \textit{paper} against \textit{rock} = \textit{paper}.\end{center}
This rock-paper-scissor game shows that an interesting extension of our model is to sample an ordered set of balls (or rather a sequence of balls) at each step and make the replacement rule depend on this sequence of $m$ randomly sampled balls.

\section{The two-colour non-balanced case (proof of Theorem~\ref{th:non-balanced})}\label{sec:non_balanced}
We have shown in this article how stochastic algorithms can be used as a powerful toolbox to study multi-drawing P\'olya urns.
Our main achievement is to remove the {\it affinity} assumption that was previously made in the literature (see~\cite{KM})
and thus considerably increase the number of urn scheme to which the strong law of large numbers and the central limit theorem apply.

The main open question is to generalise these results to the non-balanced case: 
we show in this section how to prove Theorem~\ref{th:non-balanced}.

There are two main difficulties occurring when treating the unbalanced case. Firstly,
there are only very few results in the literature on stochastic algorithms that allow the sequence $(\gamma_n)_{n\geq 0}$ to be random
(which is the case when the urn is non-balanced since $\gamma_n = \nicefrac1{T_n}$).
Secondly, we need to control the speed of convergence of the total number of balls dividing by $n$, i.e. $\nicefrac{T_n}{n}$ to its limit.

To prove Theorem~\ref{th:non-balanced}, we apply the work of Renlund~\cite{Renlund}.
In this work, a strong law of large numbers and a central limit theorem are proved, but only
in the case when the stochastic algorithm $(\theta_n)_{n\geq 0}$ is one-dimensional, 
which is why this result can only be applied to two-colour P\'olya urn schemes.

\subsection{Estimates for the total number of balls in the urn}

\begin{prop}\label{lem:cv_Tn}
If $X_n$ converges almost surely to a limit $\theta_\star$, 
then the sequence $\nicefrac{T_n}{n}$ converges almost surely to
\[\omega:=\sum_{k=0}^m \binom m k \theta_\star^k (1-\theta_\star)^{m-k} (a_{m-k}+b_{m-k}).\]
\end{prop}

\begin{proof}
In this proof, we use the notation $\mathbb E_i$ for the expectation conditionally on the canonical filtration~$\mathcal F_i$.
Write
\begin{equation}\label{eq:T_n}
\frac{T_{n+1}}{n+1} 
= \frac{T_n}{n+1} + \frac{c_{m-\zeta_{n+1}}}{n+1}
= \frac{T_0}{n+1} 
+ \frac1{n+1}\sum_{i=0}^{n} (c_{m-\zeta_{i+1}}-\mathbb E_i c_{m-\zeta_{i+1}})  
+ \frac1{n+1}\sum_{i=0}^{n} \mathbb E_i c_{m-\zeta_{i+1}},
\end{equation}
where we recall that $\zeta_{i+1}$ is the number of white balls sampled at time~$i+1$.
Note that 
\[\mathbb E_i c_{m-\zeta_{i+1}} 
= \sum_{k=0}^m \binom k i X_i^k (1-X_i)^{m-k} c_{m-k} 
\stackrel{i\to \infty}{\longrightarrow} \sum_{k=0}^m\binom k i \theta_\star^k (1-\theta_\star)^{m-k} c_{m-k} = \omega,\]
and thus, by Ces{\'a}ro's lemma, almost surely,
\[\frac1{n+1}\sum_{i=0}^{n} \mathbb E_i c_{m-\zeta_{i+1}} \stackrel{n\to \infty}{\longrightarrow} \omega.\]
Moreover, $\Delta D_{i+1} = c_{m-\zeta_{i+1}}-\mathbb E_i c_{m-\zeta_{i+1}}$ is a martingale difference sequence
such that
\[\frac{\langle D\rangle_n}{n} 
= \frac1n \sum_{i=0}^n \mathbb E_i(\Delta D_{i+1})^2 
= \frac1n \sum_{i=0}^n \Big[\mathbb E_i c^2_{m-\zeta_{i+1}}- (\mathbb E_i c_{m-\zeta_{i+1}})^2\Big].\]
Note that
\begin{align*}
\mathbb E_i c^2_{m-\zeta_{i+1}}- (\mathbb E_i c_{m-\zeta_{i+1}})^2
&= \sum_{k=0}^m \binom m k c^2_{m-k} X_i^k (1-X_i)^{m-k} - \bigg(\sum_{k=0}^m \binom m k c_{m-k} X_i^k (1-X_i)^{m-k}\bigg)^2\\
&\stackrel{i\to \infty}{\longrightarrow} \sum_{k=0}^m \binom m k c^2_{m-k} \theta_\star^k (1-\theta_\star)^{m-k} 
- \bigg(\sum_{k=0}^m \binom m k c_{m-k} \theta_\star^k (1-\theta_\star)^{m-k}\bigg)^2,
\end{align*}
thus, by Ces\'aro's lemma, $\nicefrac{\langle D\rangle_n}{n}$ converges to the same limit, and, by~\cite[Theorem~1.3.17]{Duflo}, 
we get that $\nicefrac{D_n}{n} \to 0$ almost surely when $n$ tends to infinity.
Therefore, by Equation~\eqref{eq:T_n}, $\nicefrac{T_n}{n}$ converges almost surely to $\omega$.
\end{proof}

\subsection{The associated stochastic algorithm}
As already mentioned in Section~\ref{sec:two_colour}, in the two-colour case, 
we can reduce the study to the first colour since $B_n + W_n = T_n$, and
the associated stochastic algorithm is thus one-dimensional.
\begin{lem}\label{lem:algo_sto2}
The proportion of white balls in the urn at time $n$, 
denoted by $X_n = \nicefrac{W_n}{T_n}$ satisfies the following recursion:
\[X_{n+1} = X_n + \frac1{T_{n+1}} \big(\tilde g(X_n) + \Delta M_{n+1} +\varepsilon_{n+1}\big),\]
where $\varepsilon_{n+1}$ is $\mathcal F_n$-measurable and goes almost surely to zero as $n$ goes to infinity,
\[\tilde g(x):= (1-x)\sum_{k=0}^m \binom{m}{k} a_{m-k} x^k (1-x)^{m-k}
- x \sum_{k=0}^m \binom{m}{k} b_{m-k} x^k (1-x)^{m-k},\]
and
\[\Delta M_{n+1} = Y_{n+1}- \mathbb E [Y_{n+1}|\mathcal F_n],\]
with
\[Y_{n+1} := a_{m-\zeta_{n+1}} - Z_n (a_{m-\zeta_{n+1}} + b_{m-\zeta_{n+1}}).\]
\end{lem}
The proof of this lemma is skipped since it is very similar to the proof of Lemma~\ref{lem:algo_sto}.

\subsection{Law of large numbers}
In view of Lemma~\ref{lem:algo_sto2}, one can apply the following result by Renlund~\cite{Renlund}:
\begin{thm}\label{th:renlund}
Assume that the sequence $(X_n)_{n\geq 0}$ satisfies
\begin{equation}\label{eq:algo_sto}
X_{n+1}-X_n = \gamma_{n+1}\big(f(X_n)+V_{n+1}\big),
\end{equation}
where $(\gamma_n)_{n\geq 1}$ and $(V_n)_{n\geq 1}$ are two $\mathcal F_n$-measurable sequences of random variables 
and $f$ is a continuous function from $[0,1]$ onto $\mathbb R$ such that $f(0)\geq 0$, $f(1)\leq 1$ and $f\not\equiv 0$.
Assume that, almost surely,
\[
\nicefrac{c_1}{n}\leq \gamma_n \leq \nicefrac{c_2}{n},
\qquad |V_n|\leq K_v \qquad |f(X_n)|\leq K_f, \quad \text{ and } \quad
|\mathbb E[\gamma_{n+1} V_{n+1}| \mathcal F_n]| \leq K_e \gamma_n^2,
\]
where the constants $c_1, c_2, K_v, K_f, $ and $K_e$ are positive real numbers.
Then, $\theta_\star:= \lim_{n\to +\infty} X_n$ exists almost surely, $f(\theta_\star)=0$ and $f'(\theta_\star)\leq 0$.
\end{thm}

This result allows us to prove Theorem~\ref{th:non-balanced}$(a)$:
\begin{proof}[Proof of Theorem~\ref{th:non-balanced}$(a)$]
In view of Lemma~\ref{lem:algo_sto2}, the proportion $X_n$ of white balls in the urn satisfies
\[X_{n+1} = X_n + \frac1{T_{n+1}}\big(\tilde g(X_n) + \Delta M_{n+1} + \varepsilon_{n+1}\big),\]
which is the same as Equation~\eqref{eq:algo_sto} with $f=\tilde g$,
$V_{n+1} = \Delta M_{n+1} + \varepsilon_{n+1}$, and $\gamma_{n+1} = \nicefrac1{T_{n+1}}$.
Note that, by the tenability assumption (see Lemma~\ref{lem:tenability}), 
$\tilde g(0) = a_m\geq 0$ and $\tilde g(1) = -b_0 \leq 0$ as requested in Lemma~\ref{lem:algo_sto2}.

By assumption, we have that $\liminf_{n\to \infty} \nicefrac{T_n}{n} > 0$, implying that there exists a constant $c$ such that $T_n \geq c n$ for all $n\geq 0$. If we let $C=\max_{1\leq k\leq m} c_k$, then we have almost surely
$ c n \leq T_n \leq T_0 + Cn$,
implying that there exists two constants $c_1$ and $c_2$ such that $\nicefrac{c_1}{n}\leq \gamma_n \leq \nicefrac{c_2}{n}$ 
almost surely for all $n\geq 0$.

Recall that $\Delta M_{n+1} = Y_{n+1} - \mathbb E[Y_{n+1}|\mathcal F_n]$, where $Y_{n+1} = \tilde g(X_n)$.
Note that the function $\tilde g$ is a continuous and thus bounded on $[0,1]$, 
implying that $|\Delta M_{n+1}|\leq 2 \max_{[0,1]} \tilde g$.
Also, $\varepsilon_{n+1} \to 0$ almost surely when $n$ tend to infinity, implying that this sequence is bounded. 
Thus there exists $K_v$ such that $|V_n| = |\Delta M_n + \varepsilon_n|\leq K_v$ for all $n\geq 1$.

For all $n\geq 0$, we have $|\tilde g(X_n)|\leq \max_{[0,1]} \tilde g$, and finally, 
since $T_{n+1}\geq c(n+1)$ almost surely, we have
\[|\mathbb E[\gamma_{n+1}V_{n+1}|\mathcal F_n]|
=\left|\frac1{T_n}\mathbb E\left[\Big(1-\frac{c_{\xi_{n+1}}}{T_{n+1}}\Big)\,V_{n+1}\Big|\mathcal F_n\right]\right|
=\frac{\mathbb E[|\varepsilon_{n+1}|\big|\mathcal F_n]}{T_n}
+ \frac1{T_n} \mathbb E\left[\Big|\frac{c_{\xi_{n+1}}V_{n+1}}{T_{n+1}}\Big|\,\bigg|\mathcal F_n\right],\]
where we have used that $\mathbb E [\Delta M_{n+1}|\mathcal F_n] = 0$. 
Recall that $\varepsilon_n = 0$ in the with-replacement case, 
and that $|\varepsilon_n| = \mathcal O(\nicefrac1{T_n})$ in the without-replacement case.
Also recall that, by assumption, $T_0 + cn\leq T_n\leq T_0+Cn$, implying that
there exists a constant $K_v$ such that $|\mathbb E[\gamma_{n+1}V_{n+1}|\mathcal F_n]| \leq K_v \gamma_n^2$ as requested in Lemma~\ref{lem:algo_sto2}.

We can thus apply Lemma~\ref{lem:algo_sto2}, which proves Theorem~\ref{th:non-balanced}$(a)$.
\end{proof}

\subsection{Central limit theorem}
Our aim in this section is to apply Renlund's central limit theorem for stochastic algorithms.
This result is expressed as follows: 
\begin{thm}[\cite{Renlund}]\label{th:Renlund_CLT}
Let $(X_n)_{n\geq 0}$ satisfying Equation~\eqref{eq:algo_sto} and $\theta_\star$ a stable zero of~$f$.
Conditionally on the event $\{X_n\to\theta_\star\}$,
let $\hat \gamma_n:= n\gamma_n \hat f(X_{n-1})$ where $\hat f(x) = \nicefrac{-f(x)}{x-\theta_\star}$,
and assume that $\hat \gamma_n$ converges almost surely to some deterministic limit $\hat \gamma$.
Then:
\begin{enumerate}[(i)]
\item If $\hat\gamma > \nicefrac12$ and if $\mathbb E[(n\gamma_n V_n)^2|\mathcal F_{n-1}] \to \sigma^2 > 0$,
then $\sqrt n (X_n -\theta_\star) \stackrel{n\to \infty}{\longrightarrow} \mathcal N\Big(0, \frac{\sigma^2}{2\hat\gamma -1}\Big)$, in distribution.
\item If $\hat\gamma = \nicefrac12$, $(\ln n)(\nicefrac 12 - \hat\gamma_n)\to 0$, 
and $\mathbb E[(n\gamma_n V_n)^2|\mathcal F_{n-1}] \stackrel{a.s.}{\longrightarrow} \sigma^2 > 0$ when $n\to\infty$, then
\[\sqrt{\frac n{\ln n}}(X_n -\theta_\star) \stackrel{n\to \infty}{\longrightarrow} \mathcal N(0, \sigma^2), \text{ in distribution}.\]
\end{enumerate}
\end{thm}

Renlund also proves the following result:
\begin{thm}[{\cite{Renlund}}]\label{th:renlund+}
Let $(X_n)_{n\geq 0}$ be a sequence of random variables that satisfies Equation~\eqref{eq:algo_sto}, and
$\theta_\star$ a stable zero of~$f$. 
Conditionally on $\{X_n\to\theta_\star\}$, assume that there exists $\hat\gamma\in(0,\nicefrac12)$ 
such that, almost surely when $n$ tends to infinity,
\[\hat\gamma_n - \hat \gamma = \mathcal O(|X_n-\theta_\star|+\nicefrac1n).\]
Then $n^{\hat\gamma}(X_n-\theta_\star)$ converges almost surely to a finite random variable when $n$ tends to infinity.
\end{thm}

Theorem~\ref{th:renlund}$(a)$ applies to our framework and allows us to prove Theorem~\ref{th:non-balanced}$(b)$:
\begin{proof}[Proof of Theorem~\ref{th:non-balanced}$(b)$]
The idea of the proof is to apply Renlund's result to the stochastic algorithm of Lemma~\ref{lem:algo_sto2}.
In that case,
$\hat \gamma_n = \frac{n}{T_n} \frac{\tilde g(X_{n-1})}{\theta_\star-X_{n-1}}$, $f=\tilde g$ and $V_n = \Delta M_n + \varepsilon_n$.
First note that, since $\tilde g$ is a polynomial, it is in particular differentiable in $\theta_\star$, and thus, conditionally on $\{X_n\to\theta_\star\}$, we have
\[\frac{\tilde g(X_{n-1})}{\theta_\star-X_{n-1}} \to -\tilde g'(\theta_\star).\]
Moreover, Lemma~\ref{lem:cv_Tn} gives that $\hat \gamma_n \to \nicefrac{-\tilde g'(\theta_\star)}{\omega}=:\lambda$.
Let us now have look at 
\[\mathbb E[(n \gamma_n V_n)^2 |\mathcal F_{n-1}] = \mathbb E\Big[\Big(\frac{n (\Delta M_n+\varepsilon_n)}{T_n}\Big)^2 \Big|\mathcal F_{n-1}\Big].\]
Note that $\nicefrac {T_n} n \to \omega$ almost surely, $\varepsilon_n \to 0$ almost surely, 
and $\Delta M_n = Y_{n+1}-\mathbb E[Y_{n+1}|\mathcal F_{n-1}]$.
A simple calculation (using $\tilde g(\theta_\star) = 0$) then leads to
\[\mathbb E\Big[\Big(\frac{n (\Delta M_n+\varepsilon_n)}{T_n}\Big)^2 \Big|\mathcal F_{n-1}\Big] \to \frac1{\omega^2} H(\theta_\star),\]
almost surely when $n$ goes to infinity.
We can thus apply Theorem~\ref{th:Renlund_CLT}, which concludes the proof of Theorem~\ref{th:non-balanced}$(b)$.
\end{proof}

As one can now understand, we need information about the speed of convergence of $\nicefrac{T_n}{n}$ to its limit in order to solve the two remaining cases: $\lambda = \nicefrac12$ and $\lambda <\nicefrac12$. And, in view of Theorem~\ref{th:renlund+}, it would be enough to prove that
\[\bigg|\frac{T_n}{n}-\omega\bigg| = \mathcal O\big(|Z_n-\theta_\star| + \nicefrac1n\big).\]
Unfortunately, we are not yet able to prove this statement.

\subsection{Examples}
Theorem~\ref{th:non-balanced} already permits to treat some particular cases of non-balanced P\'olya urn processes.
We give in this section examples to which our result applies, but also some examples that do not fall into our framework, either because $\lambda\leq \nicefrac12$ or because $\sigma^2 = 0$.

Note that in all examples given below, $\min_{0\leq k\leq m} c_k \geq 1$ (recall that $c_k = a_k+b_k$), 
implying in particular that $\liminf_{n\to\infty} \nicefrac{T_n}{n}>0$, as required in Theorem~\ref{th:non-balanced}. 

{\bf Example 5.1:~}
Take
\[R = \begin{pmatrix}
2 & 1\\
1 & 1\\
1 & 2
\end{pmatrix}.\]
In that case, $\tilde g(x) = (1-x)^3 - x^3$, 
and this function admits a unique zero $\theta_\star = \nicefrac12$.
This particular case thus falls under Theorem~\ref{th:non-balanced} and $X_n$ converges almost surely to $\nicefrac12$.
We also have $\omega = \nicefrac52$, and $\tilde g'(\nicefrac12) = -\nicefrac32$, implying that $\lambda = \nicefrac35>\nicefrac12$.
Finally, note that $H(\nicefrac12)=\nicefrac18$,
$\sigma^2= \nicefrac1{50}$, and $2\lambda-1 = \nicefrac15$ meaning that, in distribution,
\[\sqrt n(X_n-\nicefrac12) \stackrel{n\to\infty}{\longrightarrow} \mathcal N(0,\nicefrac1{10}).\]

{\bf Example 5.2:} Let $a$ and $b$ be two distinct integers and let $a_{m-k} = a (m-k)$ and $b_{m-k} = b k$ for all $0\leq k\leq m$. 
This example is studied in~\cite{ALS}, where it is proved that $X_n$ converges almost surely to $\frac{\sqrt{a}}{\sqrt{a}+\sqrt b}$, and that, almost surely,
\begin{equation}\label{eq:als}
\sqrt n\Big(X_n - \frac{\sqrt a}{\sqrt a +\sqrt b}\Big) \stackrel{n\to\infty}{\longrightarrow} 
\mathcal N\Big(0,\frac{\sqrt{ab}}{3m(\sqrt a+\sqrt b)^2}\Big).
\end{equation}
This result can be re-proved as a corollary of our main results.
One can check that
\[\tilde g(x) = m[a(1-x)^2-bx^2],\]
whose unique root in $[0,1]$ is $\theta_\star:=\frac{\sqrt a}{\sqrt a + \sqrt b}$.
Thus theorem~\ref{th:non-balanced}$(a)$ gives the almost sure convergence of $Z_{\infty}$ to $\theta_\star$.
Moreover,
$\tilde g'(\theta_\star) = -2m\sqrt{ab}$,
and
$\omega = m\sqrt{ab}$,
implying that $\lambda = 2$.
Finally,
\[H(\theta_\star) = \frac{m(ab)^{\nicefrac32}}{(\sqrt a+\sqrt b)^2},
\quad\sigma^2 = \frac{\sqrt{ab}}{m(\sqrt a+\sqrt b)^2} 
\quad\text{ and }\quad
2\lambda-1 = 3,\]
which, applying Theorem~\ref{th:non-balanced}$(b)$, gives Equation~\eqref{eq:als}.

\medskip
{\bf Example 5.3:}
Let $a$ and $b$ be two distinct integers and let $a_{m-k} = a k$ and $b_{m-k} = b(m-k)$ for all $0\leq k\leq m$. 
This example is studied in~\cite{ALS}, where it is proved that, if $a<b$ (resp. $a>b$) $X_n$ converges almost surely to $0$ (resp. 1) when $n$ tends to infinity, and that, almost surely
\[n^{-\nicefrac a b} X_n \to \Psi,\]
where $\Psi$ is an absolutely continuous random variable (which first and second moments are calculated).

Let us try to apply our results: one can satisfy that
\[\tilde g(x) = m(a-b)x(1-x).\]
This polynomial admits two zeros, $0$ and $1$.
Note that $\tilde g'(0) = -\tilde g'(1) = m(a-b)$.
Thus, by Theorem~\ref{th:non-balanced}$(a)$, if $a<b$, then $X_n$ converges almost surely to 0, and if $a>b$, $X_n$ converges almost surely to one.
In both cases,
\[\omega = \max (a,b) m,\]
and thus, $\lambda = \frac{|a-b|}{\max (a,b)}$.
Also note that in both cases, $H(\theta_\star) = 0$, which means that $\sigma^2 = 0$.
Even in the cases when $\lambda > \nicefrac12$, 
this example does not fall into our framework since $\sigma^2 = 0$.

{\bf Acknowledgements: }
We are very grateful to Rafik Aguech for making this collaboration possible, and for carefully proof-reading the manuscript.
A preliminary version of Theorem~\ref{th:main}$(a)$ was inaccurate in a previous version of this article; 
we warmly thank Beno{\^i}t Laslier and Jean-Fran{\c c}ois Laslier for noticing it.
CM would also like to thank the EPSRC for support through the grant EP/K016075/1.

\bibliographystyle{apt}
\bibliography{urns.bib}

\begin{thebibliography}{10}

\bibitem{ALS}
{\sc Aguech, R., Lassmar, N. and Selmi, O.}
\newblock A generalized urn model with multiple drawing and random addition.
\newblock (preprint).

\bibitem{AK}
{\sc Athreya, K.~B. and Karlin, S.} (1968).
\newblock Embedding of urn schemes into continuous time {M}arkov branching
  processes and related limit theorems.
\newblock {\em Annals of Mathematical Statistics\/} {\bf 39,} 1801--1817.

\bibitem{BDM}
{\sc Bose, A., Dasgupta, A. and Maulik, K.} (2009).
\newblock Strong laws for balanced triangular urns.
\newblock {\em Journal of Applied Probability\/} {\bf 46,} 571--584.

\bibitem{CK}
{\sc Chen, M.~R. and Kuba, M.} (2013).
\newblock On generalized {P}{\'o}lya urn models.
\newblock {\em Journal of Applied Probabilty Theory\/} {\bf 50,} 909--1216.

\bibitem{CW}
{\sc Chen, M.~R. and Wei, C.~Z.} (2005).
\newblock A new urn model.
\newblock {\em Journal of Applied Probability\/} {\bf 42,}.

\bibitem{Duflo}
{\sc Duflo, M.} (1997).
\newblock {\em Random Iterative Models}.
\newblock Springer-Verlag Berlin.

\bibitem{EP23}
{\sc Eggenberger, F. and P{\'o}lya, G.} (1923).
\newblock {\"U}ber die statistik verketetter vorg{\"a}ge.
\newblock {\em Zeitschrift f{\"u}r Angewandte Mathematik und Mechanik\/} {\bf
  1,} 279--289.

\bibitem{FDP}
{\sc Flajolet, P., Dumas, P. and Puyhaubert, V.} (2006).
\newblock Some exactly solvable models of urn process theory.
\newblock {\em DMTCS Proceedings\/} {\bf AG,} 59--118.

\bibitem{Janson04}
{\sc Janson, S.} (2004).
\newblock Functional limit theorem for multitype branching processes and
  generalized {P}\'olya urns.
\newblock {\em Stochastic Processes and their Applications\/} {\bf 110,}
  177--245.

\bibitem{Janson06}
{\sc Janson, S.} (2006).
\newblock Limit theorems for triangular urn schemes.
\newblock {\em Probability Theory and Related Fields\/} {\bf 134,} 417--452.

\bibitem{KMP}
{\sc Kuba, M., Mahmoud, H. and Panholzer, A.} (2013).
\newblock Analysis of a generalized {F}riedman's urn with multiple drawings.
\newblock {\em Discrete Applied Mathematics\/} {\bf 161,} 2968--2984.

\bibitem{KM}
{\sc Kuba, M. and Mahmoud, H.~M.} (2017).
\newblock Two-colour balanced affine urn models with multiple drawings.
\newblock {\em Advances in Applied Mathematics\/} {\bf 90,} 1--26.

\bibitem{KS}
{\sc Kuba, M. and Sulzbach, H.} (2016).
\newblock On martingale tail sums in affine two-color urn models with multiple
  drawings.
\newblock {\em Journal of Applied Probability\/}.
\newblock to appear.

\bibitem{LP}
{\sc Laruelle, S. and Pag\`es, G.} (2013).
\newblock Randomized urn models revisited using stochastic approximation.
\newblock {\em Annals of Applied Probability\/} {\bf 23,} 1409--1436.

\bibitem{LL}
{\sc Laslier, B. and Laslier, J.-F.} (2013).
\newblock Reinforcement learning from comparisons: Three alternatives is
  enough, two is not.
\newblock \texttt{ArXiV:1301.5734}.

\bibitem{Basile}
{\sc Morcrette, B.} (2012).
\newblock Fully analyzing an algebraic {P}\'olya urn model.
\newblock In {\em Latin American Symposium on Theoretical Informatics (LATIN)}.
\newblock vol.~7256.
\newblock LNCS.
\newblock pp.~568--581.

\bibitem{MM}
{\sc Morcrette, B. and Mahmoud, H.} (2012).
\newblock Exactly solvable balanced tenable urns with random entries via the
  analytic methodology.
\newblock In {\em 23rd International Meeting on Probabilistic, Combinatorial,
  and Asymptotic Methods for the Analysis of Algorithms (AofA)}.
\newblock vol.~AQ.
\newblock DMTCS.
\newblock pp.~219--232.

\bibitem{Pemantle}
{\sc Pemantle, R.} (2007).
\newblock A survey of random processes with reinforcement.
\newblock {\em Probability surveys\/} {\bf 4,} 25.

\bibitem{Renlund}
{\sc Renlund, H.} (2011).
\newblock Limit theorems for stochastic approximation algorithms.
\newblock \texttt{ArXiV:1102.4741v1}.

\bibitem{Zhang16}
{\sc Zhang, L.-X.} (2016).
\newblock Central limit theorems of a recursive stochastic algorithm with
  applications to adaptive design.
\newblock {\em Annals of Applied Probability\/} {\bf 26,} 3630--3658.

\end{thebibliography}
\end{document}